\documentclass[11pt,twoside]{article}

\usepackage{subcaption} 
\usepackage{amssymb, geometry}
\usepackage{pdflscape}
\usepackage{amsthm, amsfonts, amsmath, latexsym}
\usepackage{xfrac}
\usepackage{epsfig,graphicx,xcolor}
\usepackage{rotating}
\usepackage{natbib}
\usepackage[noresetcount,lined,boxed]{algorithm2e} 
\usepackage{refcheck}           
\usepackage{url}                
\usepackage{version}            
\usepackage{pgf,pgfarrows,pgfnodes,pgfautomata,pgfheaps,pgfshade}
\usepackage[affil-it]{authblk}
\usepackage{tikz,pgfplots,pgfplotstable}
\usepackage{booktabs,colortbl}

\usepackage[colorlinks]{hyperref}
\hypersetup{
  colorlinks=true,
  linkcolor=blue,
  citecolor=blue,
}

\graphicspath{{figures/}}    

\newcommand{\pagebudget}[1]{} 

\newtheorem{proposition}{Proposition}[section]

\newcommand{\R}{\mathbb{R}}

\newcommand{\Z}{\mathbb{Z}}
\newcommand{\mini}{\mathop{\mbox{minimize}}}

\newcommand{\st}{\mbox{subject to }}

\newcommand{\be}{\begin{equation}}
\newcommand{\ee}{\end{equation}}
\newcommand{\bea}{\begin{eqnarray}}
\newcommand{\eea}{\end{eqnarray}}
\newcommand{\dps}{\displaystyle}

\newcommand{\pref}[1]{(\ref{#1})}
\newcommand{\bvec}{\left(\begin{array}{c}}
\newcommand{\evec}{\end{array}\right)}
\newcommand{\bsub}{\begin{subequations}}
\newcommand{\esub}{\end{subequations}}

\newcommand{\ds}{\displaystyle}

\newcommand{\cT}{{\cal T}} 
\newcommand{\cJ}{{\cal J}} 
\newcommand{\cD}{{\cal D}} 
\newcommand{\cP}{{\cal P}} 
\newcommand{\cU}{{\cal U}}

\newcommand{\xb}{{\bar x}} 
\newcommand{\vb}{{\bar v}} 
\newcommand{\wb}{{\bar w}} 
\newcommand{\Xb}{{\bar X}}
\newcommand{\Vb}{{\bar V}}
\newcommand{\Wb}{{\bar W}}

\newcommand{\nn}{\nonumber}
\newcommand{\tc}{\textcolor}

\definecolor{darkgreen}{RGB}{0,100,0}   
\begin{document}

\newpage
\pagestyle{plain}
\pagenumbering{roman}
\setcounter{page}{2}


\vfill

\newpage
\pagestyle{plain}
\setcounter{page}{1}
\pagenumbering{arabic}


\norefnames
\nocitenames

\pagestyle{myheadings}
\markboth{Fu Lin, Sven Leyffer, and Todd Munson}
         {Two-Level MILP}

\author{Fu Lin, Sven Leyffer, and Todd Munson\thanks{Mathematics and Computer Science Division, Argonne National Laboratory, Argonne, IL 60439 (fulin@mcs.anl.gov, leyffer@mcs.anl.gov, tmunson@mcs.anl.gov).}}
\title{A Two-Level Approach to Large Mixed-Integer Programs with 
\\[0.25cm]
Application to Cogeneration in Energy-Efficient Buildings\thanks{Preprint ANL/MCS-P5332-0415}}

\date{\today}

\maketitle

\begin{abstract}
We study a two-stage mixed-integer linear program (MILP) with more than 1 million binary variables in the second stage. We develop a two-level approach by constructing a semi-coarse model (coarsened with respect to variables) and a coarse model (coarsened with respect to both variables and constraints). We coarsen binary variables by selecting a small number of pre-specified daily on/off profiles. We aggregate constraints by partitioning them into groups and summing over each group. With an appropriate choice of coarsened profiles, the semi-coarse model is guaranteed to find a feasible solution of the original problem and hence provides an upper bound on the optimal solution. We show that solving a sequence of coarse models converges to the same upper bound with proven finite steps. This is achieved by adding violated constraints to coarse models until all constraints in the semi-coarse model are satisfied. We demonstrate the effectiveness of our approach in cogeneration for buildings. The coarsened models allow us to obtain good approximate solutions at a fraction of the time required by solving the original problem. Extensive numerical experiments show that the two-level approach scales to large problems that are beyond the capacity of state-of-the-art commercial MILP solvers.

  


\end{abstract}


\section{Problem Definition and Motivation}
 \label{S:Intro}

We consider a hierarchical two-stage mixed-integer linear program (MILP) with integer variables in both the first and second stages. We are particularly interested in applications where the first-stage integer variables model design or purchasing decisions and the second-stage variables model operational decisions over a long time horizon (e.g., hourly operations over a decadal horizon). The goal is to take operational constraints into account when making a capital investment decision. Thus, our model is complicated by the fact that second-stage variables include both binary (on/off) decisions and continuous variables that model operational settings. Examples include the design of cogeneration units for commercial buildings subject to operational conditions~\citep{sidmarbailac05,stamarsidlaicofaki09,rengao10,prubranew13a,pruleynewbra14} and transmission network expansion subject to unit commitment constraints~\citep{algmotcon03,hedferonefisore10,munsauhob13,munhobwat14}. Models of this class can involve hundreds of thousands or even millions of binary variables and are beyond the scope of today's state-of-the-art solvers.

We consider a two-stage MILP with $m$ first-stage variables and three sets of $N$ second-stage variables. The first-stage variables are $y \in \{0,1\}^m$. We have three classes of second-stage variables: (1) on/off decisions, $x \in \{0,1\}^N$; (2) operational settings $v \in \R^N$ that are switched on or off by $x$, that is, $ L x \leq v \leq U x$ for finite bounds $L \leq U$; and (3) other second-stage variables $0 \leq w \in \R^N$.  
The MILP model is described by
\be\label{E:MILP}
   \begin{array}{ll}
     \dps \mini_{y,x,v,w} & a^T y + b^T x + c^T v + d^T w \\
     \st                 & A y + B x + C v + D w \leq f \\
                         & y \in \{ 0,1\}^m     \\
                         & x \in \{ 0,1\}^N     \\
                         & v \in \R^N, \; L x \leq v \leq U x  \\
                         & w \in \R^N, \; w \geq 0       ,      
   \end{array}
\ee
where $a \in \R^m$, $b,c,d \in \R^N$, $A \in \R^{M \times m}$, and $B, C, D \in \R^{M \times N}$. We assume that $m \ll N$ and $m \ll M$; that is, the number of first-stage variables is much smaller than the number of second-stage binary variables and coupling constraints.

In addition to a large number of binary and continuous variables in the second stage, the challenges of model~\eqref{E:MILP} also arise from coupling constraints. We do not assume any sparsity structure in matrices $B$, $C$, and $D$. This is in contrast to the arrow-type sparsity structure that arises, for example, in stochastic programs~\citep{birlou11}. Therefore, fixing the first-stage variable $y$ does not decompose~\eqref{E:MILP} into scenario-based subproblems. Moreover, the coupling constraints in~\eqref{E:MILP} are not amenable to decomposition methods such as Lagrangian relaxation~\citep{geo74,fis81}. The reason is that the number of coupling constraints, $M$, is of the same order as the number of variables, $N$. Dualizing all coupling constraints results in a large number of dual variables; hence, Lagrangian relaxation of~\eqref{E:MILP} does not yield efficient decomposition methods~\citep{fis81}.

Problem~\eqref{E:MILP} arises in several applications. In particular, our work is motivated by the cogeneration problem with renewable energy for commercial buildings. In this case, the first-stage design involves investment decisions for cogeneration units such as fuel cells, solar panels, and battery storage. The second-stage problem aims at optimal on/off hourly operation that takes into account technology specifics such as minimum/maximum power generation. Our goal is to include operational constraints in the design of cogeneration.

One of our objectives is to find the optimal first-stage solution of~(\ref{E:MILP}). However, the two-stage MILP with a large number of binary variables at the second stage is beyond the scope of state-of-the-art commercial MIP solvers. For example, a typical cogeneration model with a ten-year horizon results in $1.05$ million binary variables ($3650$ days $\times$ $24$ hours $\times$ $12$ units). On the other hand, a naive approach that solves~\eqref{E:MILP} with a short horizon at the second stage provides first-stage designs that are suboptimal for a long horizon problem. The reason is that short horizon problems do not take into account coupling constraints over a long horizon. Moreover, the problem data of short-horizon problems are not representative of long-horizon problems, resulting in suboptimal solutions.

We develop a {\em two-level approach\/} that coarsens the hourly on/off variables to daily operation profiles. Since the profile representation yields a model with many fewer variables, we refer to this step as the {\em primal (variable) coarsening\/}. The resulting {\em semi-coarse model\/} still contains the same order of constraints as in~\eqref{E:MILP}. We reduce the number of constraints by partitioning them into groups that are of the same size as the profiles and summing over each group. This aggregation of constraints results in a relaxed problem whose solutions may not be feasible for the original MILP model. We include the violated constraints, re-solve the coarsened MILP model, and repeat this process until all constraints are satisfied. We refer to the aggregation of constraints as the {\em dual (constraint) coarsening\/} and the resulting MILP as the {\em coarse model\/}. 


\subsection{Literature Review}

The idea of aggregating variables and constraints to build approximate optimization models is not new. Zipkin studies the effect of variable aggregation and row aggregation for linear programs in~\citep{zip80a,zip80b}. This framework is extended to stochastic linear programs by~\citet{bir85} and by~\citet{clagro97}. For integer programs, a large amount of work focuses on aggregating constraints into one or more {\em surrogate\/} constraints for which one can show that the solution of the original problem and that of the surrogate problem are identical. Early work includes that of~\citet{bal65}, \citet{glo68}, \citet{geo69}, and \citet{glo77}; see also the survey paper by~\citet{rogplawoneva91}. 

We note that the theory of surrogate constraints focuses mainly on integer programs with non-negative integer coefficients. The seminal work by~\citet{mat96} results in an exponential grow of the coefficients in the surrogate constraints. Many refinements have been developed, including simultaneous and sequential aggregation schemes, to reduce the magnitude of coefficients~\citep{rogplawoneva91}. For mixed-integer programs with real-valued coefficients, however, we are not aware of aggregation schemes that guarantee equivalence between the original problem and the surrogate problem. In our two-level approach, we employ constraint aggregation as a relaxation technique to identify a smaller set of constraints that are active at the optimal solution. This is achieved by solving a sequence of MILPs and adding the violated constraints until all constraints in the original MILP are satisfied. Furthermore, we take advantage of the LP warm-start to reduce the number of MILP re-solves and the computational time of each MILP.

A large body of literature appeared in the 1990s on bilevel or multilevel mixed integer programs; see~\citep{moobar90,wenyan90,edmbar92,hanjausav92,wenhua96} and survey papers~\citep{wenhsu91,viccal94}. One of the main theoretical emphases is on decentralized decision-making from a game-theoretic point of view~\citep{wenhsu91}. Optimality conditions for convex bilevel programs have been established in~\citep{viccal94}. A heuristic-based branch-and-bound method is developed by~\citet{wenyan90} and \citet{hanjausav92}, and tabu search method is introduced by~\citet{wenhua96}. Recent years have seen specific application-driven algorithms ranging from infrastructure protection planning~\citep{scachu08} to vulnerability analysis of power grids~\citep{pinmezdonles10}.

In this context, the multilevel method is specifically referred to the problem formulation. In contrast, our two-level method concentrates on algorithmic development for mixed-integer linear programs. In particular, our two-level approach builds optimization models of different resolutions from a fine-level problem to a coarse-level problem. We develop a systematic procedure that coarsens binary and continuous variables in addition to the aggregation of constraints. Our approach allows us to solve large MILPs with more than one million binary and continuous variables and coupling constraints. Extensive numerical experiments have been conducted to verify the efficiency of the developed algorithm. 

Our two-level approach is reminiscent of multi-grid methods in linear algebra and solution of partial differential equations. Related work includes that of~\citet{gelman90}, who discussed a multilevel iterative method for generic optimization problem; \citet{grasartoi08}, who developed a recursive trust-region method for nonlinear unconstrained problems; and \citet{wengol09}, who proposed a line search multi-grid approach for nonlinear programs. 


\subsection{Our Contributions}

Our contributions are fourfold. First, we develop a systematic two-level approach for two-stage MILPs with a large number of binary and continuous variables in the second stage. The coarsening of binary variables is done by introducing profiles (vectors) with binary elements. By selecting one profile from a small number of candidate profiles, we reduce the number of binary variables by orders of magnitude. In addition, we reduce the number of continuous variables by using a convex combination of profiles with real elements. We show that the coarsening in variables results in a tightening of the original MILP and therefore provides an upper bound on the optimal solution. 

Second, we devise a simple scheme for constraint aggregation that does not require any sparsity assumptions on the coefficient matrices. By partitioning constraints into groups and summing up constraints in each group, we obtain a relaxation of the original MILP with many fewer constraints. We solve a sequence of MILPs by adding any violated constraints until all constraints are satisfied. We show that the LP-relaxation of MILPs can be used to warm-start the MILP re-solves and significantly reduce the computational effort. 

Third, we implement our two-level approach in the context of the design of cogeneration for buildings. The resulting complex MILP model has a large number of coupling constraints over a long time horizon. We employ a moving-horizon method to generate valid profiles. We show by construction that the generated profiles satisfy coupling constraints in time, hence further reducing the number of constraints in the coarsened models.

Fourth, we verify the efficiency of our algorithm on a variety of examples generated from simulation programs for commercial buildings. While our two-level approach requires solving a sequence of coarse MILPs, numerical results indicate that the first iterate provides good approximate solutions of the semi-coarse models. Through extensive numerical experiments, we demonstrate the scalability of the two-level approach on a rich set of large problems that are beyond the capacity of state-of-the-art commercial solvers. 

\paragraph{Outline.} The remainder of this paper is organized as follows. In Section~\ref{sec.multilevel}, we describe the two-level approach for the two-stage MILPs. We show that our approach results in a tightening of the original problem and provides an upper bound on the optimal solution. In Section~\ref{sec.cogen}, we apply our approach to a complex MILP model from the cogeneration problem in buildings. We discuss how the profiles are generated and selected such that the coarsened models provide feasible solutions to~\eqref{E:MILP}. In Section~\ref{sec.results}, we demonstrate the effectiveness of our approach on a diverse set of test problems. We show that the two-level approach allows us to find good approximate solutions with a fraction of computational time compared with solving the full model. In Section~\ref{sec.concl}, we summarize our contribution and discuss future extensions. 

\section{Two-Level Approach to MILP}
   \label{sec.multilevel}

In this section, we derive two models that are formulated in terms of 
the first-stage variables and $n \ll N$ second-stage variables. We start by 
presenting our ideas for variable coarsening, resulting in a semi-coarse model
that has ${\cal O}(n)$ second-stage variables and ${\cal O}(N)$ second-stage constraints. Next we show how to further coarsen this model by aggregating constraints, resulting in a coarse model with both ${\cal O}(n)$ second-stage variables and constraints.


\subsection{Primal Coarsening}

To coarsen the variables in the model, we introduce a coarsening factor $\delta \in \Z_+$
and group $N$ second-stage variables into $n = N/\delta$ equal-sized groups of size 
$\delta$ (the assumption that the groups are equal-sized is not critical): 
\begin{equation}
   x_i = \bvec x_{\delta(i-1)+1} \\ \vdots \\ x_{\delta i} \evec
   \quad \mbox{for } \; i=1, \ldots, n.
\end{equation}
We define groups for $v$ and $w$ analogously.
These groups correspond to a partition of the second-stage variables $x,v$, and $w$:
\begin{equation}
   x = \bvec x_1 \\  \vdots \\  x_n \evec, \quad
   v = \bvec v_1 \\  \vdots \\  v_n \evec, \quad
   w = \bvec w_1 \\  \vdots \\  w_n \evec.
\end{equation}
Next, we introduce $\delta$-profiles, which are fixed parameter values for
a group of variables:
\begin{equation} 
   \Xb_k \in \{0,1\}^{\delta} \; \mbox{ for } \; k=1,\ldots,K.
\end{equation}
We show in Section~\ref{sec.cogen} how these profiles are generated and selected. Figure~\ref{fig.dailyprofile} illustrates that $N=48$ hourly on/off variables can be represented by $n=2$ daily profiles with length $\delta=24$. Now for every $\Xb_k$, we collect a set of $I_k$ operational $\delta$-profiles for $\Vb$ and $\Wb$:
\begin{equation}
\label{eqn.V}
   \Vb_{jk} \in \R^{\delta}, \mbox{ with } \; L \Xb_k \leq \Vb_{jk} \leq U \Xb_k
   \; \mbox{ for } \; j=1,\ldots, I_k, \;  k=1,\ldots,K,
\end{equation}
and
\begin{equation}
    0 \leq \Wb_j \in \R^{\delta} \; \mbox{ for } \; j=1,\ldots, J,
\end{equation}
respectively.

\begin{figure}
  \centering
  \begin{tabular}{c}
  \begin{tikzpicture}
    \begin{axis}[ybar interval, width=0.7\textwidth, height=0.15\textheight, xmin=0, xmax=50, xtick=\empty, ytick={0,1}, grid=major, ylabel=On/off]
    \addplot[color=red, fill=pink, line width = 1pt] table [x = time, y = onoff] {twodayonoff.tab};
    \end{axis}
  \end{tikzpicture}  
    \\
  \begin{tikzpicture}
    \begin{axis}[width=0.7\textwidth, height=0.15\textheight, no markers, xmin=0, xmax=50, xtick={0,24,48}, ytick={0,1}, grid=major, ylabel=Profiles, xlabel=Hours]
    \addplot[const plot, color=darkgreen, line width = 1pt] table [x = time, y = base] {twodayonoff.tab};
    \end{axis}
  \end{tikzpicture}  
  \end{tabular}
  \caption{The first row shows the hourly on/off variables $x_t \in \{0,1\}$ for $48$ hours. The second row shows the profile representation with two daily profiles $\Xb_k \in \{0,1\}^{24}$.}
  \label{fig.dailyprofile}
\end{figure}
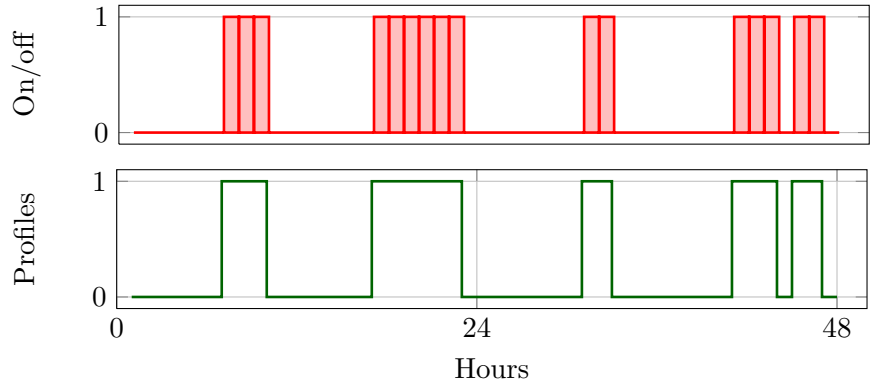

Given these sets of $\delta$-profiles, we perform a change of variables
that replaces the second-stage variables $(x,v,w)$ by a reduced set of
variables $(\xb,\vb,\wb)$ for the $\delta$-profiles,
resulting in a coarsening of the second-stage variables. In particular, we set
\begin{equation}\label{E:Xk}
   x_i = \sum_{k=1}^K \xb_{ik} \Xb_k \; , \quad  \sum_{k=1}^K \xb_{ik} \leq 1
   \; , \quad  \xb_{ik} \in \{0,1\} \quad \mbox{ for }\; i=1,\ldots, n, \; k=1,\ldots, K.
\end{equation}
Thus, we have that $\xb \in \{0,1\}^{Kn}$, and our goal is to create models where $Kn \ll N$.
Similarly, we write $v$ and $w$ as
\begin{equation}\label{E:Vk}
  v_i =  \sum_{k=1}^K \sum_{j=1}^{I_k} \vb_{ijk} \Vb_{jk} \, , 
  \quad \sum_{k=1}^K \sum_{j=1}^{I_k} \vb_{ijk} \leq 1 \, ,
  \quad \sum_{j=1}^{I_k} \vb_{ijk} = \xb_{ik} \, ,
  \quad 0 \leq \vb_{ijk} \leq 1 \, ,
  \quad \mbox{ for } i=1,\ldots, n,
\end{equation}
and 
\begin{equation}\label{E:Wk}
  w_i =  \sum_{j=1}^J \wb_{ij} \Wb_j \; , 
  \quad  \sum_{j=1}^J \wb_{ij} \leq 1  \; ,
  \quad \mbox{ for }\; i=1,\ldots, n.
\end{equation}
We note that we do not enforce $\wb_{ij} \in \{0,1\}$ or $\vb_{ijk} \in \{0,1\}$. The reason is that we wish to allow more freedom for choosing operational profiles in the second stage as long as they remain feasible. This choice also simplifies the coarsened second-stage problem and provides a valid approach if we assume that convex combinations of operational $\delta$-profiles $\Vb_{jk}$ and $\Wb_j$ are feasible, which is often the case in practice.

The partition of variables $(x,v,w)$ implies a partition of the problem matrices $B$, $C$, and $D$ of \eqref{E:MILP} as
\[
   B = [ B_1 : \cdots : B_{n} ], \quad
   C = [ C_1 : \cdots : C_{n} ], \quad
   D = [ D_1 : \cdots : D_{n} ],
\]
where $B_i$, $C_i$, and $D_i \in \R^{M \times \delta}$ for $i=1,\ldots,n$.
Next, we define $\delta$-profile matrices as
\begin{equation} \label{E:profs} 
  \Xb = [ \Xb_1 : \cdots : \Xb_K ], \quad
  \Vb = [ \Vb_{11} : \cdots : \Vb_{1I_1} : \ldots :  \Vb_{K 1} : \cdots : \Vb_{K I_K} ], \quad
  \Wb = [ \Wb_1 : \cdots :  \Wb_J],
\end{equation}
where $\Xb_k$, $\Vb_{jk}$, and $\Wb_j$ are vectors of length $\delta$. We define aggregated (coarse) matrices $\bar{B}$, $\bar{C}$, and $\bar{D} \in \R^{M \times n \delta}$ as
\begin{equation}
\label{E:aggMatrix}
   \bar{B} = [ B_1 \Xb : \cdots : B_{n} \Xb ], \quad
   \bar{C} = [ C_1 \Vb : \cdots : C_{n} \Vb ], \quad
   \bar{D} = [ D_1 \Wb : \cdots : D_{n} \Wb ] ,
\end{equation}
and aggregated cost vectors $\bar{b},\bar{c},\bar{d} \in \R^{Kn}$ as
\begin{equation}
  \label{E:aggVector}
   \bar{b} = [ b_1^T \Xb: \cdots : b_n^T \Xb ]^T, \quad
   \bar{c} = [ c_1^T \Vb: \cdots : c_n^T \Vb ]^T, \quad
   \bar{d} = [ d_1^T \Wb: \cdots : d_n^T \Wb ]^T.  
\end{equation}

We obtain the following aggregated MILP, which we refer to as the {\em semi-coarse MILP},
because it has been coarsened in primal variables only:
\begin{equation}\label{E:aggMILP}
   \begin{array}{ll}
     \dps \mini_{y,\xb,\vb,\wb} & a^T y + \bar{b}^T \xb + \bar{c}^T \vb + \bar{d}^T \wb \\
     \st                 & A y + \bar{B} \xb + \bar{C} \vb + \bar{D} \wb \leq f \\
                         & \dps  \sum_{k=1}^K \xb_{ik} \leq 1 \; , \quad  \xb_{ik} \in \{0,1\} \\
                         & \dps  \sum_{k=1}^K \sum_{j=1}^{I_k} \vb_{ijk} \leq 1 \; , \quad \sum_{j=1}^{I_k} \vb_{ijk} = \xb_{ik} \, , 
                            \quad  0 \leq \vb_{ijk} \leq 1  \\
                         & \dps  \sum_{j=1}^{J} \wb_{ij} \leq 1 \; , \quad  0 \leq \wb_{ij} \leq 1.
   \end{array}
\end{equation}
The MILP \pref{E:aggMILP} has potentially many fewer variables than \pref{E:MILP},
but contains as many constraints as the original problem. Before we show how to aggregate 
constraints in the next section, we finish this section by summarizing 
some of the properties of \pref{E:aggMILP}.

\begin{proposition} \label{P:aggMILP-feas}
 Let $(\xb,\vb,\wb)$ be a feasible point of the semi-coarse model~\pref{E:aggMILP}, then it follows
    that the corresponding fine-scale variables
    \begin{equation}
    \label{E:finevars}
       x = \bvec 
             \dps \sum_{k=1}^K \xb_{1k} \Xb_k \\ \vdots \\ \dps \sum_{k=1}^K \xb_{nk} \Xb_k
           \evec
           , \quad
       v = \bvec 
             \dps \sum_{k,j=1}^{K,I_k} \vb_{1jk} \Vb_{jk} \\ \vdots \\ \dps \sum_{k,j=1}^{K,I_k} \vb_{njk} \Vb_{jk}
           \evec
           , \quad
       w = \bvec 
             \dps \sum_{j=1}^J \wb_{1j} \Wb_j \\ \vdots \\ \dps \sum_{j=1}^J \wb_{nj} \Wb_j
           \evec
    \end{equation}
    are feasible in the original MILP~\pref{E:MILP}.
\end{proposition}
\begin{proof}
The binary constraint $x \in \{0,1\}^N$ follows from the representation~(\ref{E:Xk}) and the definition of binary vectors $\Xb_k$. Similarly, the non-negativity of $w$ is a direct consequence of non-negative profiles $\Wb_j$ and non-negative coefficients $\wb_{ij}$. To show that $v$ is feasible, we start with the definition of $\Vb_{jk}$ in~\eqref{eqn.V}:
\[
L \cdot \Xb_k  \leq  \Vb_{jk}  \leq  U \cdot \Xb_k.
\]
Since $\vb_{ijk} \geq 0$, it follows that 
\[
   L \sum_{k=1}^K \sum_{j=1}^{I_k} \vb_{ijk} \Xb_k 
     \leq 
   \sum_{k=1}^K \sum_{j=1}^{I_k} \vb_{ijk} \Vb_{jk} 
     \leq 
   U \sum_{k=1}^K \sum_{j=1}^{I_k} \vb_{ijk} \Xb_k.
\]
Using $\xb_{ik}=\sum_{j=1}^{I_k} \vb_{ijk}$, we have
\[
L \sum_{k=1}^K \xb_{ik} \Xb_k  
     \leq  v_i  \leq 
U \sum_{k=1}^K \xb_{ik} \Xb_k, \quad \mbox{ for } i=1,\ldots,n.
\]
Using the definition of $x_i$ in \eqref{E:finevars}, it follows that $L x_i \leq v_i \leq U x_i$, for $i=1,\ldots,n$, and thus $L x \leq v \leq U x$. The proof is complete by noting that 
\[
f \geq Ay + \bar{B} \xb + \bar{C} \vb + \bar{D} \wb = Ay + Bx + Cv + Dw,
\]
where the equality follows from the definition of aggregated matrices~(\ref{E:aggMatrix}) and the representation of the fine-scale variables~(\ref{E:finevars}).
\end{proof}

Next, we show that the semi-coarse model provides an upper bound.

\begin{proposition}
  \label{P:aggMILP-upper}
The semi-coarse model~\pref{E:aggMILP} is a tightening of the original MILP~\pref{E:MILP}, and its solution provides an upper bound on~\pref{E:MILP}. The two problems are equivalent if the optimal profiles from the solution of~\pref{E:MILP} are included in~\pref{E:aggMILP}.
\end{proposition}
\begin{proof}
Let $z^\star_{\rm MILP}$ and $z^\star_{\rm semi}$ be the optimal value of the original MILP~(\ref{E:MILP}) and the semi-coarse model~\eqref{E:aggMILP}, respectively. Since any feasible point of the semi-coarse model~(\ref{E:aggMILP}) is a feasible point of the original MILP~(\ref{E:MILP}), it follows that~(\ref{E:aggMILP}) is a tightening of~(\ref{E:MILP}) and thus
\[
z^\star_{\rm MILP}\leq z^\star_{\rm semi} \, . 
\]
Let $(x^\star,v^\star,w^\star)$ be the optimal solution of MILP~\eqref{E:MILP}. We now extract optimal profiles $(\Xb^\star,\Vb^\star,\Wb^\star)$ corresponding to $(x^\star,v^\star,w^\star)$. Then it follows that there exists a solution of~\eqref{E:aggMILP} such that 
\begin{equation}
  \nonumber
x_i^\star = \sum_{k=1}^K \xb^\star_{ik} \Xb_k^\star, \quad
v_i^\star = \sum_{k=1}^K \sum_{j=1}^{I_k} \vb^\star_{ijk} \Vb_{jk}^\star, \quad
w_i^\star = \sum_{j=1}^J \wb^\star_{ij} \Wb_j^\star, 
\quad \mbox{ for } i=1,\ldots,n.
\end{equation}
Moreover, solution $(\xb^\star,\vb^\star,\wb^\star)$ is feasible in \eqref{E:aggMILP}, and the objective value is the same as solution $(x^\star,v^\star,w^\star)$. Because (\ref{E:aggMILP}) is a tightening of (\ref{E:MILP}), there cannot be a better solution than $(\xb^\star,\vb^\star,\wb^\star)$ that we constructed.  Hence, $(\xb^\star,\vb^\star,\wb^\star)$ is the optimal solution, and both problems (\ref{E:MILP}) and (\ref{E:aggMILP}) are equivalent.
\end{proof}

\subsection{Dual Coarsening}

We next coarsen the constraints by partitioning them into $n$ groups of 
size $\delta$ and summing over each group. We define the rows of $A$ as $A_i^T$,
and we define a set of aggregated constraints by summing over the rows of $A$ within
each group:
\[ \hat{A} := \bvec 
                 \dps \sum_{i=1}^{\delta} A_{i}^T \\
                 \dps \sum_{i=1}^{\delta} A_{\delta + i}^T \\
                 \vdots \\
                 \dps \sum_{i=1}^{\delta} A_{\delta (n-1) + i}^T 
                \evec .
\]
We aggregate the constraint matrices $\bar{B}$, $\bar{C}$, and $\bar{D}$ in a similar way
\[ \hat{\bar{B}} := \bvec 
                 \dps \sum_{i=1}^{\delta} \bar{B}_{i}^T \\
                 \dps \sum_{i=1}^{\delta} \bar{B}_{\delta + i}^T \\
                 \vdots \\
                 \dps \sum_{i=1}^{\delta} \bar{B}_{\delta (n-1) + i}^T 
                \evec , \quad
    \hat{\bar{C}} := \bvec 
                 \dps \sum_{i=1}^{\delta} \bar{C}_{i}^T \\
                 \dps \sum_{i=1}^{\delta} \bar{C}_{\delta + i}^T \\
                 \vdots \\
                 \dps \sum_{i=1}^{\delta} \bar{C}_{\delta (n-1) + i}^T 
                \evec  , \quad
    \hat{\bar{D}} := \bvec 
                 \dps \sum_{i=1}^{\delta} \bar{D}_{i}^T \\
                 \dps \sum_{i=1}^{\delta} \bar{D}_{\delta + i}^T \\
                 \vdots \\
                 \dps \sum_{i=1}^{\delta} \bar{D}_{\delta (n-1) + i}^T 
                \evec .
\]
Note that we use ``bar'' notation to denote coarsening in variables and ``hat'' notation to denote coarsening in constraints. We coarsen the right-hand side $f$ analogously and obtain the {\em coarse MILP}:
\begin{equation}\label{E:coarseMILP}
   \begin{array}{ll}
     \dps \mini_{v,w,x,y} & a^T y + \bar{b}^T \xb + \bar{c}^T \vb + \bar{d}^T \wb \\
     \st                 & \hat{A} y + \hat{\bar{B}} \xb + \hat{\bar{C}} \vb + \hat{\bar{D}} \wb \leq \hat{f} \\
                         & \dps \sum_{k=1}^K \xb_{ik} \leq 1 \; , \quad  \xb_{ik} \in \{0,1\} \\
                         & \dps \sum_{k=1}^K \sum_{j=1}^{I_k} \vb_{ijk} \leq 1 \; , \quad \sum_{j=1}^{I_k} \vb_{ijk} = \xb_{ik} \, , 
                            \quad  0 \leq \vb_{ijk} \leq 1 \\
                         & \dps  \sum_{j=1}^{J} \wb_{ij} \leq 1 \; , \quad  0 \leq \wb_{ij} \leq 1.
   \end{array}
\end{equation}

It follows easily that \pref{E:coarseMILP} is a relaxation of \pref{E:aggMILP},
because we have simply aggregated the constraints. We can use this fact to 
develop a simple algorithm that solves \pref{E:aggMILP} by solving a sequence of
tighter relaxations. The main idea is that after solving a relaxation  \pref{E:coarseMILP}
we can check whether all constraints in \pref{E:aggMILP} are satisfied and add any
violated constraints to \pref{E:coarseMILP}. It follows easily that this algorithm is finite, 
because after finitely many constraints have been added to \pref{E:coarseMILP},
it is equivalent to \pref{E:aggMILP}. 

In practice, however, solving a sequence of MILPs \pref{E:coarseMILP} may not be efficient, because MILPs do not warm-start. Instead, we solve the LP-relaxation of the coarse model~\eqref{E:coarseMILP} and add violated constraints until all constraints in the LP-relaxation of the semi-coarse model~\eqref{E:aggMILP} are satisfied. We use the identified constraints as the initial set of constraints for the MILP coarse model iterations. We summarize this procedure in Algorithm~\ref{alg.coarseMILP}.

\begin{algorithm}[htb]
\SetAlgoVlined
\begin{center}
{\bf LP Warm-start Phase}
\end{center}
Set $l \gets 0$\;
Solve the LP-relaxation of the coarse MILP~\pref{E:coarseMILP} to get solution $(y^0_{\rm LP},\xb^0_{\rm LP}, \vb^0, \wb^0)$\;
\While{$(y^l_{\rm LP},\xb^l_{\rm LP}, \vb^l, \wb^l)$ is not feasible for the LP-relaxation of the semi-coarse MILP \pref{E:aggMILP}}{
  Find constraints that are violated in the LP-relaxation of \pref{E:aggMILP} and add them to \pref{E:coarseMILP}\;
  Solve the LP-relaxation of \pref{E:coarseMILP} with the added constraints to get $(y^{l+1}_{\rm LP},\xb^{l+1}_{\rm LP}, \vb^{l+1}, \wb^{l+1})$\;
  Set $l = l+1$ \;
} 
\begin{center}
{\bf MILP Phase}  
\end{center}
Set $k \gets 0$\;
Solve the coarse MILP~\pref{E:coarseMILP} with additional constraints added in the {\bf LP Warm-start Phase} to get solution $(y^0,\xb^0, \vb^0, \wb^0)$\;
\While{$(y^k,\xb^k, \vb^k, \wb^k)$ is not feasible for \pref{E:aggMILP}}{
  Find constraints that are violated in \pref{E:aggMILP} and add them to \pref{E:coarseMILP}\;
  Solve the coarse MILP~\pref{E:coarseMILP} with the added constraints to get $(y^{k+1},\xb^{k+1}, \vb^{k+1}, \wb^{k+1})$\;
  Set $k = k+1$ \;
} 
\caption{Solve the semi-coarse MILP \pref{E:aggMILP} via a sequence of coarse MILPs \pref{E:coarseMILP} using the LP-relaxation as warm-start.}
\label{alg.coarseMILP}
\end{algorithm}

\begin{proposition}
  \label{P:coarseMILP}
When Algorithm \ref{alg.coarseMILP} terminates, the solution $(y^k,\xb^k, \vb^k, \wb^k)$ of the coarse model \pref{E:coarseMILP} with added constraints is the solution of the semi-coarse model \pref{E:aggMILP}. 
\end{proposition}
\begin{proof}
Since $(y^k,\xb^k, \vb^k, \wb^k)$ minimizes the objective function of the semi-coarse model~\pref{E:aggMILP} and satisfies all constraints~\pref{E:aggMILP}, it is the optimal solution of~\pref{E:aggMILP}.
\end{proof}

We show in Section~\ref{sec.results} that our approach in Algorithm~\ref{alg.coarseMILP} is advantageous; in particular, it significantly reduces the number of MILP re-solves, as opposed to the case without LP-relaxation as warm-start. We conclude this section by providing a matrix representation of the two-level approach using Kronecker products. Recall that for two generic matrices $E \in \R^{n\times m}$ and $F \in \R^{p\times q}$, their Kronecker product is defined as
\[
E \otimes F 
 = 
\left[
  \begin{array}{ccc}
    e_{11} F & \cdots & e_{1m} F \\
    \vdots  & \ddots & \vdots  \\
    e_{n1} F & \cdots & e_{nm} F
  \end{array}
\right]
\in \R^{np \times mq}.
\]
Using \pref{E:profs}, we can write 
\[ \bar{B} = B (I \otimes \Xb), \quad \bar{C} = C (I \otimes \Vb), \quad \bar{D} = D (I \otimes \Wb) \]
where $I$ is the identity matrix. Similarly, if we let ${\bf 1} = (1,\ldots,1) \in \R^{\delta}$ be the row vector of all ones of length $\delta$, then we can express
\[ \hat{\bar{B}} = (I \otimes {\bf 1}) \bar{B} =  (I \otimes {\bf 1}) B (I \otimes \Xb), \quad 
   \hat{\bar{C}} = (I \otimes {\bf 1}) \bar{C} =  (I \otimes {\bf 1}) C (I \otimes \Vb), \quad 
   \hat{\bar{D}} = (I \otimes {\bf 1}) \bar{D} =  (I \otimes {\bf 1}) D (I \otimes \Wb),
\]
and 
\[ \hat{A} = (I \otimes {\bf 1}) A, \quad \hat{f} = (I \otimes {\bf 1}) f. \]

\section{Application to Cogeneration for Buildings}
  \label{sec.cogen}

In this section, we apply our two-level approach to the cogeneration problem for buildings. Our MILP model~\eqref{eqn.model} is adapted from models for cogeneration in commercial buildings~\citep{rengao10, prubranew13a}. In particular, we take linearized models for fuel cells and water tank storage from the work of~\citet{prubranew13a}, and we penalize on/off operations using switching cost as done by~\citet{rengao10}. While the two-level framework described in Section~\ref{sec.multilevel} applies to generic MILPs~\eqref{E:MILP}, the MILP model~\eqref{eqn.model} for cogeneration entails several complex constraints as discussed in Section~\ref{sec.semicoarse}. As a result, additional work is required to construct appropriate semi-coarse and coarse models. 

We note that our MILP model~\eqref{eqn.model} for the cogeneration problem has the following features that are not included in existing models. First, our model has a large number of binary variables, on the order of ${\cal O}(10^6)$, in the second stage. This is orders of magnitude larger than the number of binary variables of~\citet{sidmarbailac05}, \citet{stamarsidlaicofaki09}, \citet{rengao10}, and \citet{prubranew13a}. This modeling feature allows considerably more degrees of freedom for on/off operation during the life time of new technologies. Second, our model contains three sets of coupling constraints that (i) couple first- and second-stage binary variables, (ii) couple second-stage variables for different technologies, and (iii) couple second-stage variables over a long time horizon (e.g., 10-20 years). The number of coupling constraints is on the same order of variables, namely, ${\cal O}(10^6)$. This makes the problem significantly harder because it does not lend itself to decomposition techniques such as Lagrangian relaxation. 

In what follows, we describe the main characteristics of the MILP model (Section~\ref{sec.model}), derive the semi-coarse and coarse models (Sections~\ref{sec.semicoarse} and~\ref{sec.coarse}), and discuss profile generation and selection (Section~\ref{sec.profgen}).

\subsection{MILP Model} 
 \label{sec.model}

The cogeneration problem consists of two components: the investment decision and the operation planning. The investment decision concerns what new technologies to purchase, while the operation planning concerns how to dispatch units over a long-term period (e.g., $10$-$20$ years). We formulate a two-stage MILP, where the first-stage variables model investment decisions and the second-stage variables model equipment operations. Note that both the first and second stages contain integer variables. In particular, on/off operations for new technologies in second-stage are made on an hourly basis.

Following~\citet{sidmarbailac05, stamarsidlaicofaki09, rengao10, prubranew13a}, we consider five technologies: batteries, boilers, solid-oxide fuel cells (SOFCs), combined heat and power (CHP) SOFCs, and water tank storage. The selection criterion is based on the installation, operation, maintenance, fuel consumption, and carbon emission costs, while meeting electricity and heating demands. The detailed MILP model is given in~\eqref{eqn.model}.  

\subsection{Semi-Coarse Model} \label{sec.semicoarse}

We begin by introducing the profiles for variables. We denote profiles by upper-case letters with the bar notation on top: $\bar{X}_{jk} \in \cP_o$ are the on/off profiles for technology $j$, $\bar{P}_{jkl} \in \cP_p$ are the production profiles associated with $\bar{X}_{jk}$, $\bar{U}_k \in \cP_u$ are the profiles for power purchased from the utility, $\bar{Q}_k \in \cP_q$ are the heat generation profiles from boilers, $\bar{B}_k, \bar{B}_k^{\rm IO} \in \cP_b$ are the power storage and input/output profiles for batteries, and $\bar{S}_k, \bar{S}_k^{\rm out} \in \cP_s$ are the heat storage and output profiles for water tanks, where $\cP_o, \cP_p, \cP_u, \cP_q, \cP_b$, and $\cP_s$ are the corresponding set of profiles. We denote the $h$th element of a profile by $\bar{X}_{jk}(h)$ for $h \in \{1,\ldots,\delta\}$. Time-dependent parameters $D_t^P$ and $D_t^Q$ and discount factors $Y_t$ are concatenated into profiles $\bar{D}_d^P,\bar{D}_d^Q$, and $\bar{Y}_d \in \R^{\delta}$. 

As described in Section~\ref{sec.multilevel}, the variables in the original MILP model are represented by using profiles in the coarsened models. In particular, binary variables $x_{ijt}$ are coarsened to profiles $\bar{X}_{jk}$, and continuous variables $p_{ijt},u_t,b_t,s_t,s_t^{\rm out}$, and $q_{t}$ are coarsened to profiles $\bar{P}_{jkl},\bar{U}_k,\bar{B}_k,\bar{S}_k$, and $\bar{S}_k^{\rm out},\bar{Q}_k \geq 0$, respectively.  The coefficients for profiles are denoted by lower-case letters with a bar on top; for example, $\bar{x}_{ijdk} \in \{0,1\}$ indicates whether profile $k$ is selected on day $d$ for unit $i$ of technology $j$. As modeled in \eqref{eqn.setppcx}, no more than one on/off profile can be selected for fixed $i,j$, and $d$. This, in conjunction with binary elements of $\bar{X}_{jk}$, results in binary-valued $x_{ijt}$. Similarly, non-negative profiles  $\bar{P}_{jkl},\bar{U}_k,\bar{B}_k,\bar{S}_k,\bar{S}_k^{\rm out}$, and $\bar{Q}_k \geq 0$ and their non-negative coefficients in \eqref{eqn.setppcv}, \eqref{eqn.boundsw}, and \eqref{eqn.setppcw} result in non-negative $p_{ijt},u_t,b_t,s_t,s_t^{\rm out}$, and $q_{t}$.

We next turn to constraints that  {\em do not\/} couple variables in time, namely, \eqref{eqn.PowerDemand}, \eqref{eqn.OnOff}, \eqref{eqn.symbrk}, \eqref{eqn.maxpower}, \eqref{eqn.batterybnd}, \eqref{eqn.HeatDemand}, \eqref{eqn.storbnd}, and \eqref{eqn.boilbnd}. The profile representation for these constraints is straightforward; one simply substitutes the summation of profiles as shown in \eqref{eqn.semi-powerdemand}, \eqref{eqn.semi-cap-sofc}, \eqref{eqn.semi-symbrk}, \eqref{eqn.semi-maxdemand}, \eqref{eqn.semi-cap-batt}, \eqref{eqn.semi-HeatDemand}, \eqref{eqn.semi-cap-stor}, and \eqref{eqn.semi-cap-boil}, respectively. Note that all inequalities in \eqref{eqn.semicoarse} are elementwise inequalities. We do not include the production constraint \eqref{eqn.powergen} in the semi-coarse model. Instead, we require that the production profile $\bar{P}_{jkl}$ associated with the on/off profile $\bar{X}_{jk}$ satisfy 
\begin{equation}
  \label{eqn.OnOffProf}
  R_j^{\rm min} \bar{X}_{jk} \leq \bar{P}_{jkl} \leq R_j^{\rm max} \bar{X}_{jk}.
\end{equation}
This approach is justified because \eqref{eqn.powergen} follows by construction due to the convex combination of coefficients for production profiles in \eqref{eqn.setppcv}. By an analogous argument, $s_t^{\rm out} \leq s_{t}$ in \eqref{eqn.storbnd} is a direct consequence of the following condition imposed on profiles:
\begin{equation}
  \label{eqn.storageBND}
  \bar{S}_k^{\rm out} \leq  \bar{S}_k. 
\end{equation}

We next discuss how to coarsen variables that are coupled over time periods. The boundary conditions~\eqref{eqn.batteryboundary} and \eqref{eqn.storbnd} can be expressed as \eqref{eqn.semi_batt} and \eqref{eqn.semi-heatbnd}. The maximum power purchased constraint \eqref{eqn.maxpower} can be rewritten as \eqref{eqn.semi-maxdemand}, where ${\bf 1}$ denotes the vector of all ones with length $\delta$. We next turn to switching and storage constraints \eqref{eqn.switching}, \eqref{eqn.battery}, and \eqref{eqn.storage}, whose profile representation needs additional notation. Given an on/off profile $\bar{X}_{jk} \in \{0,1\}^\delta$, we can construct the associated switching profile 
\begin{equation}
\label{eqn.switchprof}
\bar{W}_{jk}(h)  =  
|\bar{X}_{jk}(h+1) - \bar{X}_{jk}(h)|.
\end{equation}
Now, the switching cost can be included in the objective function in~\eqref{eqn.semicoarse}. To deal with the battery storage constraint \eqref{eqn.battery} that couples $b_t$ and $b_t^{\rm IO}$ over the horizon $T$, we require that the pair of profiles $(\bar{B}_k$, $\bar{B}_k^{\rm IO})$ satisfy
\begin{equation}
  \label{eqn.batteryIO}
  \bar{B}_k(h+1) = (1-L^P) \bar{B}_k(h) + \bar{B}_k^{\rm IO}(h), \quad h=1,\ldots,\delta-1.
\end{equation}
We assign the same coefficient $\bar{b}_{dk}$ to both sets of profiles $\bar{B}_k$ and $\bar{B}_k^{\rm IO}$ throughout~\eqref{eqn.semicoarse}. It follows that \eqref{eqn.battery} is satisfied except for hours between profiles, namely, $t = d \delta$ for $d < |\cD|$. Therefore, we introduce constraint \eqref{eqn.semi-midnights} to guarantee that \eqref{eqn.battery} holds for $t = d \delta$ for $d < |\cD|$. We will show in Section~\ref{sec.profgen} how to generate profiles that satisfy~\eqref{eqn.OnOffProf}, \eqref{eqn.storageBND}, and \eqref{eqn.batteryIO}. The heat storage constraint \eqref{eqn.storage} for $j={\rm chp}$ can be expressed as 
\[
            \sum_{k \in \cP_s} \bar{s}_{dk} 
            \left\{ 
              \big[ S_\delta - (1-L^Q)I_\delta \big] \bar{S}_{k}
              +  \bar{S}^{\rm out}_{k}
            \right\} 
            + \sum_{k \in \cP_s} \bar{s}_{(d+1)k} E_\delta \bar{S}_{k}       
            \leq (E_{j}^Q / E_{j}^P) \sum_{i,k,l} \bar{p}_{ijdkl} \bar{P}_{jkl}, ~ d < |\cD|,
\]
where $I_\delta \in \R^{\delta \times \delta}$ is the identity matrix, $E_{\delta1} \in \R^{\delta \times \delta}$ is a matrix with $1$ in the $(\delta,1)$ entry, and $S_\delta \in \R^{\delta \times \delta}$ is a Toeplitz matrix with only nonzero elements being $1$ at the first upper-subdiagonal. 

\subsection{Coarse Model} \label{sec.coarse}

We next turn to the coarse model. Note that constraints in the semi-coarse model \eqref{eqn.semicoarse} are elementwise equalities or inequalities involving daily profiles. We aggregate hourly constraints by summing the elements of daily profiles. Using the hat notation on the top to denote the element-wise summation of profiles, for example, $\hat{\bar{P}}_{jkl} = \sum_{h=1}^\delta \bar{P}_{jkl}(h)$, we coarsen $\delta$ hourly constraints into a single constraint. For example, the hourly demand constraint~\eqref{eqn.semi-powerdemand} is replaced by the daily demand constraint~\eqref{eqn.coarse-powerdemand}. Similarly, the maximum hourly power demand \eqref{eqn.semi-maxdemand} is aggregated into the maximum daily demand \eqref{eqn.coarse-maxdemand}. The symmetric breaking constraint \eqref{eqn.coarse-symmbrk} has the interpretation that the number of times unit $i+1$ is turned on is no greater than the number of times unit $i$ is turned on in day $d$. The capacity constraints for technologies \eqref{eqn.coarse-cap-sofc}, \eqref{eqn.coarse-cap-batt}, \eqref{eqn.coarse-cap-stor}, and \eqref{eqn.coarse-cap-boil} imply that the daily power/heat output is bounded by the daily capacity of the technology units purchased at the first-stage. Since \eqref{eqn.semi-midnights}, \eqref{eqn.semi_batt}, and \eqref{eqn.semi-heatbnd} are scalar constraints themselves, no aggregation is applied to \eqref{eqn.coarse-batt-midnight}, \eqref{eqn.coarse-batt-bnd}, \eqref{eqn.coarse-stor-bnd} in the coarse model. Also, since the coefficients of profiles are the same for both semi-coarse and coarse models, \eqref{eqn.coarse-setppcx}, \eqref{eqn.coarse-setppcv}, and \eqref{eqn.coarse-setppcw} stay the same as \eqref{eqn.setppcx}, \eqref{eqn.setppcv}, and \eqref{eqn.setppcw} in the semi-coarse model.

\subsection{Profile Generation and Selection} \label{sec.profgen}

Recall that the on/off profiles $\bar{X}_{jk}$ must have binary elements and the profiles $\bar{P}_{jkl}$, $\bar{Q}_k$, $\bar{U}_k$, $\bar{B}_k$, $\bar{S}_k$, and $\bar{S}^{\rm out}_k$ must have non-negative elements. In addition, we impose constraints \eqref{eqn.OnOffProf}, \eqref{eqn.storageBND}, and \eqref{eqn.batteryIO} in formulating the semi-coarse model in Section~\ref{sec.semicoarse}. In this section, we discuss how to generate valid profiles, and we provide suggestions for profile selections.

One approach to generating profiles that satisfy the above constraints is as follows. We solve a number of small instances of the original MILP \eqref{eqn.model}, take snapshots of the second-stage solutions, and extract profiles from these snapshots. For example, consider a four-day MILP~\eqref{eqn.model}; that is, the time horizon in the second-stage problem has only four days. Solving such a four-day model, we have four snapshots of daily operation and production; in particular, we have four sets of on/off profiles for $x_{ijt} \in \{0,1\}$ and four sets of profiles $p_{ijt}$, $u_t$, $b_t$, $s_t$, $s_t^{\rm out}$, and $q_{jt} \geq 0$. We will show in Proposition~\ref{pro.movinghorizon} that these are valid profiles for the semi-coarse model~(\ref{eqn.semicoarse}).

The remaining question is how to choose the short-horizon MILPs. Our objective is to generate a rich set of profiles that are representative of the optimal solutions for a long-horizon MILP \eqref{eqn.model}. To this end, we borrow the {\em moving-horizon approach\/} from model predictive control~\citep{garpremor89,allzhe00}. Given a long-horizon MILP, the idea is to solve MILPs over a short window, roll the window forward, and re-solve the new MILP until the window reaches the end of the horizon; see Figure~\ref{fig.rollinghorizon} for an illustration. We summarize this approach and provide additional details in Algorithm~\ref{alg.movinghorizon}. 

\begin{algorithm}[htb]
 \KwData{The parameters for MILP \eqref{eqn.model} with a horizon of $D$ days, and a window of $w$ days with $w \ll D$.} 
 \KwResult{Profiles $\{\bar{P}_{jkl}$, $\bar{Q}_k$, $\bar{U}_k$, $\bar{B}_k$, $\bar{S}_k$, $\bar{S}^{\rm out}_k \}$ $\in \R_+^\delta$, $\bar{B}_k^{\rm IO} \in \R^\delta$, and $\{\bar{X}_{jk}, \bar{W}_{jk}\} \in \{0,1\}^\delta$.}
\SetAlgoVlined
Set $k \gets 0$, ${\cal W} = \{ 1, \ldots, \delta w \}$, ${\cal R} = \{ 1, \ldots, \delta \}$\;
\While{$k+w \leq D$}{
  Solve MILP~\pref{eqn.model} in the current window $t \in {\cal W}$\;
  Take snapshots of solutions $\{x_{ijt}$, $p_{ijt}$, $q_{t}$, $u_t$, $b_t$, $b_t^{\rm IO}$, $s_t$, $s_t^{\rm out} \}$ in $t \in {\cal R}$\; 
  Extract profiles 
$\{ \bar{X}_{jk}$,  $\bar{P}_{jkl}$, $\bar{Q}_k$, $\bar{U}_k$, $\bar{B}_k$,  $\bar{B}_k^{\rm IO}$, $\bar{S}_k$, $\bar{S}^{\rm out}_k \}$,
and $\bar{W}_{jk} \in \{0,1\}^\delta$ using \eqref{eqn.switchprof}\; 
  Roll the window forward by setting $k \gets k+1$\;
  {\bf if}~ $k+w < D$ ~{\bf then}~ Set ${\cal W} = \{ \delta k + 1, \ldots, \delta (k+w) \}$, ${\cal R} = \{ \delta k + 1, \ldots, \delta (k+1) \}$\;
  {\bf if}~ $k+w = D$ ~{\bf then}~ Set ${\cal W} = {\cal R} = \{ \delta k + 1, \ldots, \delta H \}$.
} 
\caption{Moving horizon method for profile generation.}
\label{alg.movinghorizon}
\end{algorithm}

\begin{figure}
  \centering
      \includegraphics[width=0.75\textwidth]{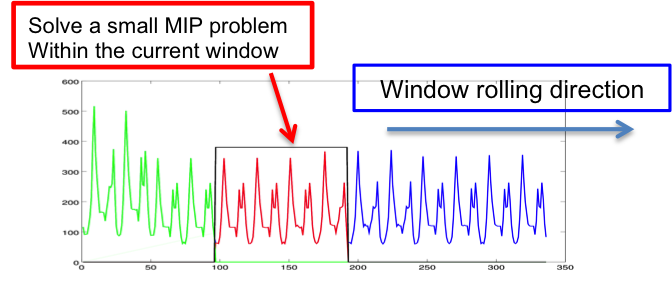}        
  \caption{Illustration of the moving horizon method in Algorithm~\ref{alg.movinghorizon}.}
  \label{fig.rollinghorizon}
\end{figure}

\begin{proposition}
  \label{pro.movinghorizon}
 The moving-horizon method described in Algorithm~\ref{alg.movinghorizon} generates non-negative profiles $\{ \bar{P}_{jkl}$, $\bar{Q}_k$, $\bar{U}_k$, $\bar{B}_k$, $\bar{S}_k$, $\bar{S}^{\rm out}_k \}$ $\in \R_+^\delta$ and binary profiles $\{\bar{X}_{jk}, \bar{W}_{jk}\} \in \{0,1\}^\delta$. Moreover, the profile pairs $\{\bar{X}_{jk},\bar{P}_{jkl}\}$, $\{\bar{S}_k^{\rm out}, \bar{S}_k\}$, and $\{\bar{B}_k, \bar{B}_k^{\rm IO}\}$ satisfy \eqref{eqn.OnOffProf}, \eqref{eqn.storageBND}, and \eqref{eqn.batteryIO}, respectively.
\end{proposition}
\begin{proof}
Non-negativity of the profiles $\{\bar{P}_{jkl}$, $\bar{Q}_k$, $\bar{U}_k$, $\bar{B}_k$, $\bar{S}_k$, $\bar{S}^{\rm out}_k\}$ follows from the fact that the second-stage solutions $\{p_{ijt}$, $q_{t}$, $u_t$, $b_t$, $s_t$, $s_t^{\rm out}\}$ of the original MILP model~(\ref{eqn.model}) are non-negative. Similarly, $\bar{X}_{jk} \in \{0,1\}^\delta$ follows from $x_{ijt} \in \{0,1\}$; and $\bar{W}_{jk}$, constructed from~\eqref{eqn.switchprof}, is elementwise binary. Since the production variable $p_{ijt}$ and the on/off variable $x_{ijt}$  satisfy \eqref{eqn.powergen}, it follows that $\{\bar{X}_{jk},\bar{P}_{jkl}\}$ satisfies \eqref{eqn.OnOffProf}. Since the power storage $b_t$ and power input/output $b_t^{\rm IO}$ satisfy \eqref{eqn.battery} and since the heat storage $s_t$ and heat output $s_t^{\rm out}$ satisfy \eqref{eqn.storage}, we conclude that \eqref{eqn.storageBND} and \eqref{eqn.batteryIO} follow by construction. 
\end{proof}

Algorithm~\ref{alg.movinghorizon} potentially generates a huge number of profiles; hence, solving~\eqref{eqn.semicoarse} and \eqref{eqn.coarse} with all generated profiles may be prohibitive. Instead, we select on/off profiles $\bar{X}_{jk}$ that appear most frequently in the generated profile pool. Since the aim of the two-level approach is to reduce the problem size, it is desired that the number of on/off profiles $k$ is much smaller than the length of profiles $\delta$ (e.g., $k \approx \delta/10$). For the production profiles $\bar{P}_{jkl}$ associated with each $\bar{X}_{jk}$, we choose production profiles that have the minimum or maximum total production $\sum_{h=1}^\delta \bar{P}_{jkl}(h)$. These extreme profiles provide an envelope of other profiles; thus, their convex combination~\eqref{eqn.setppcv} provides a good range of profiles for selection. Similarly, profiles with the minimum and maximum sum of absolute values are chosen for the battery storage $\bar{B}_k$, the heat storage $\bar{S}_k$, the battery input/output $\bar{B}^{\rm IO}_k$, and the heat output $\bar{S}_k^{\rm out}$. Further, profiles for the power purchased from the utility $\bar{U}_k$ and the heat output from the boiler $\bar{Q}_k$ are uniformly sampled over the period of the entire horizon.

Alternatively, we also employ a {\em k-means clustering algorithm\/} in order to cluster profiles. Given a prespecified $k$ number of clusters, this algorithm assigns profiles to one of $k$ clusters defined by the centroids~\citep{jai10}. Since demand and pricing data for buildings tend to differ significantly in winters and summers, we set $k=2$, and choose one profile that has the minimum least-squares distance to the centroids. In our numerical experiments in Section~\ref{sec.results}, we apply a $k$-means algorithm to the boiler output $\bar{Q}_k$, power purchased from the utility $\bar{U}_k$, heat storage output $\bar{S}_k^{\rm out}$, and battery power output $\bar{B}^{\rm IO}_k$ profiles. The clustering approach achieves better performance in the objective function than does simple sampling heuristics for profile selection.

\section{Numerical Results and Case Studies} \label{sec.results}

In this section, we illustrate the effectiveness of our two-level approach using five different building examples. We demonstrate that both the semi-coarse and the coarse models allow us to find good approximate solutions in a fraction of the time compared with solving the full MILP model. The two-level approach also scales to large problems that are beyond the scope of state-of-the-art commercial MILP solvers. 

\subsection{Generation of Problem Instances} 

We use the simulation program {\em EnergyPlus 8.2.\/} \citep{EnPlus01,crapedlawwin00} to generate hourly electricity and heating demands for typical summer and winter days for five types of buildings located in Chicago, Illinois. We scale the daily demands, depending on building types, to generate weekly summer and winter demands. The scaling factors for building types are shown in Table~\ref{tab.building}. We generate yearly demand profiles by interpolating between winter and summer weeks. The demand of the $i$th week is given by $D_i = w_i D_{\rm wt} + (1-w_i) D_{\rm sm}$, where $D_{\rm wt}$ and $D_{\rm sm}$ are the weekly demand in winter and summer, respectively, and $w_i$ is a piecewise linear function shown in Figure~\ref{fig.linfun}. The resulting yearly demand profiles are perturbed by a zero mean unit variance Gaussian noise whose magnitude is $2\%$ proportional to the magnitude of demands. We repeat this process to generate demands for multiple years.

Figure \ref{fig.Electr} shows the electricity demand in summer and winter weeks for five different buildings: an office building, a supermarket, a full-service restaurant, a hospital, and a stand-alone retail store. In addition to different peak demand patterns, the number of kilowatt-hours varies by orders of magnitude for different buildings, from tens of kilowatt-hours for the restaurant to more than a thousand kilowatt-hours for the hospital. The diversity of peak demand patterns and the wide range of magnitudes show the richness of the example set. As a result, the difficulty of the MILP model~\eqref{eqn.model} varies significantly for different buildings; see Section~\ref{sec.fullsol}.

We follow the pricing structure of~\citet{stamarsidlaicofaki09}. The price for electricity and gas is $\$0.12$ per kWh and $\$0.049$ per kWh, respectively. The peak demand charge is $\$14.2$ per kW for summer months (from June to September) and $\$11.36$ per kW for winter months (from October to May). These prices are fixed for each year over the second-stage time horizons.

\begin{table}
  \centering
  \begin{tabular}{|c|ccccccc|}
    \hline
    Building Type      & Mon. & Tue. & Wed. & Thu. & Fri. & Sat. & Sun. \\
    \hline
    Office             & 1    & 1    & 1    & 1    & 1    & 1/2  & 1/4  \\
    Supermarket        & 1/2  & 1/2  & 1/2  & 1/2  & 1/2  & 1    & 1    \\
    Restaurant         & 1/4  & 1/4  & 1/2  & 1/2  & 1    & 1    & 1/2  \\
    Hospital           & 1    & 1    & 1    & 1    & 1    & 1    & 1    \\
    Retail             & 1/4  & 1/2  & 1/2  & 1    & 1    & 1/4  & 1/4  \\
    \hline
  \end{tabular}
  \caption{Scaling factors of daily demands in a week for five buildings. For example, the demand of the office building on Saturdays is scaled by half of the demand on weekdays.}
  \label{tab.building}
\end{table}

\begin{figure}
  \centering
  \begin{tikzpicture}
    \begin{axis}[width=0.75\textwidth, height=0.15\textheight, xmin=0, xmax=52, ytick={0,1}, xtick={0,13,24,37,44,52}, grid=major, ylabel=weight $w_i$, xlabel=weeks]
    \addplot[line width = 1pt] table [x = weeks, y = weights] {linfun.tab};
    \end{axis}
  \end{tikzpicture}    
  \caption{Piecewise linear interpolation between winter and summer weeks. The demand of the $i$th week is given by $D_i = w_i D_{\rm wt} + (1-w_i) D_{\rm sm}$, where $D_{\rm wt}$ is the weekly demand in winter and $D_{\rm sm}$ is the weekly demand in summer.}
  \label{fig.linfun}
\end{figure}
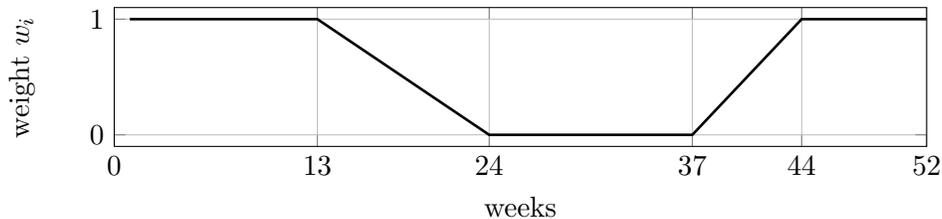

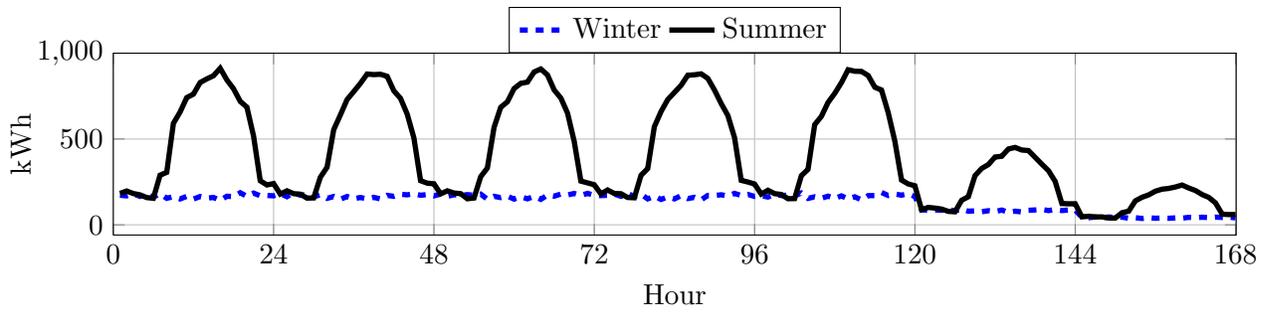
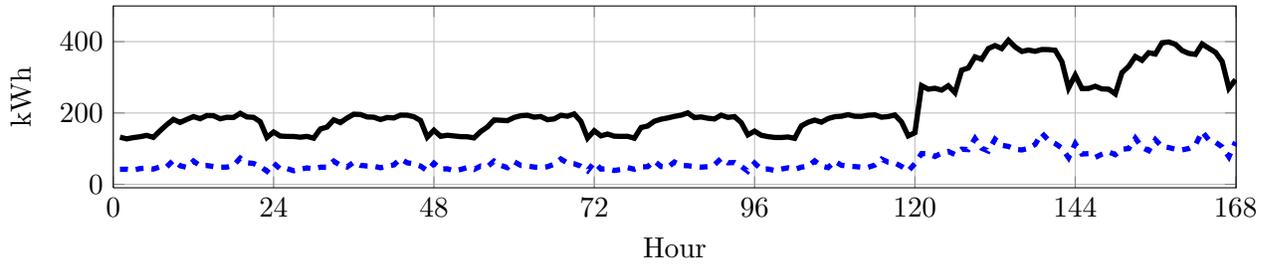
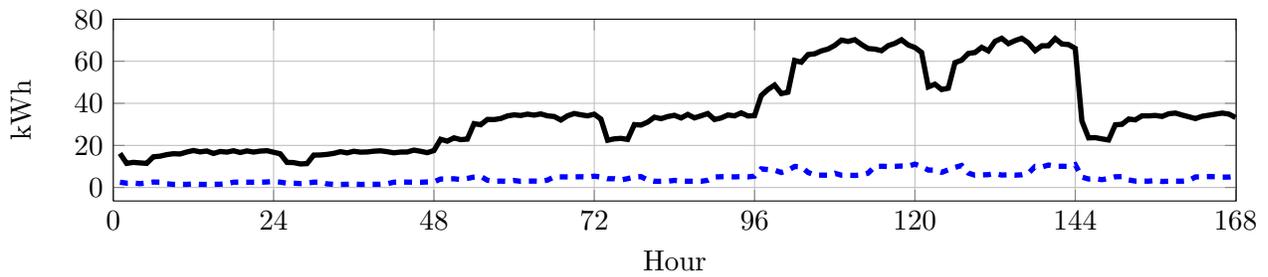
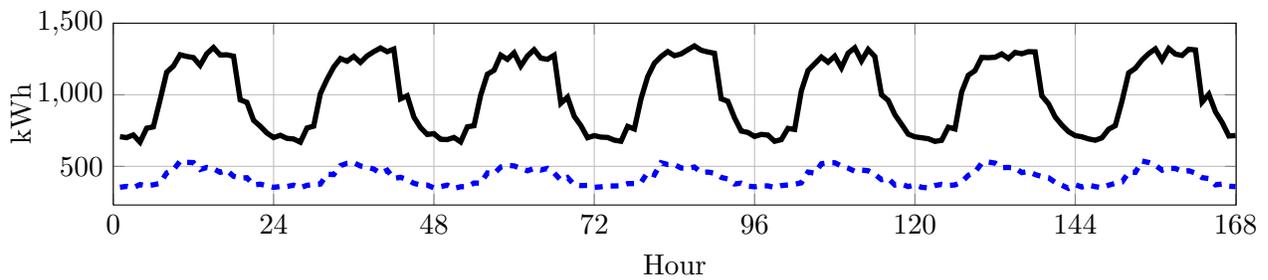
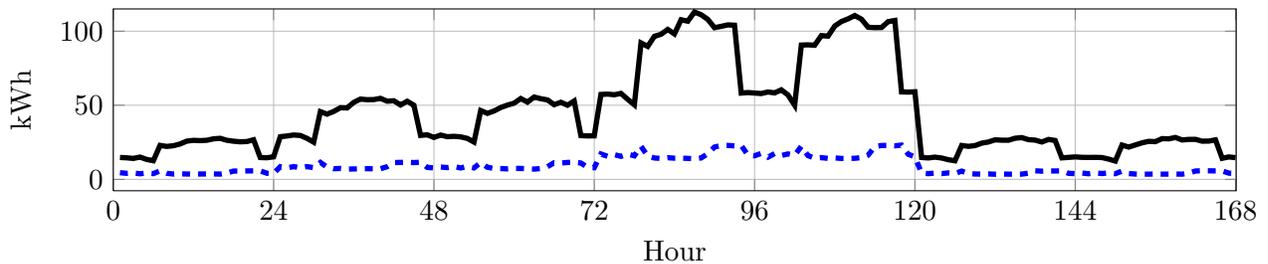
\begin{figure}
\centering
\begin{subfigure}{\textwidth}
  \begin{tikzpicture}
    \begin{axis}[width=\textwidth, height=0.175\textheight, xmin=0, xmax=168.1, ymax=1000, grid=major, xlabel=Hour, ylabel=kWh, xtick={0,24,48,72,96,120,144,168},legend style={at={(0.5,1.0)}, anchor=south,legend columns=-1}]
    \addplot[color=blue, dashed, line width = 2pt] table [x = Time, y = ElectrWinter] {SchoolWeek.tab};
    \addplot[color=black, line width = 2pt] table [x = Time, y = ElectrSummer] {SchoolWeek.tab};    
    \legend{Winter, Summer}
    \end{axis}
  \end{tikzpicture}
  \caption{Electricity demand for an office building in winter and summer weeks.}
  \label{fig.schoolElectr}
\end{subfigure}
\\[0.2cm]
\begin{subfigure}{\textwidth}
  \begin{tikzpicture}
    \begin{axis}[width=\textwidth, height=0.175\textheight, xmin=0, xmax=168.1, ymax=500, grid=major, xlabel=Hour,  ylabel=kWh, xtick={0,24,48,72,96,120,144,168},legend style={at={(0.5,1.0)}, anchor=south,legend columns=-1}]
    \addplot[color=blue, dashed, line width = 2pt] table [x = Time, y = ElectrWinter] {MarketWeek.tab};
    \addplot[color=black, line width = 2pt] table [x = Time, y = ElectrSummer] {MarketWeek.tab};    
    \end{axis}
  \end{tikzpicture}
  \caption{Electricity demand for a supermarket in winter and summer weeks.}
  \label{fig.marketElectr}
\end{subfigure}
\\[0.2cm]
\begin{subfigure}{\textwidth}
  \begin{tikzpicture}
    \begin{axis}[width=\textwidth, height=0.175\textheight, xmin=0, xmax=168.1, ymax=80, grid=major, xlabel=Hour, ylabel=kWh, xtick={0,24,48,72,96,120,144,168},legend style={at={(0.5,1.0)}, anchor=south,legend columns=-1}]
    \addplot[color=blue, dashed, line width = 2pt] table [x = Time, y = ElectrWinter] {RestaurantWeek.tab};
    \addplot[color=black, line width = 2pt] table [x = Time, y = ElectrSummer] {RestaurantWeek.tab};    
    \end{axis}
  \end{tikzpicture}
  \caption{Electricity demand for a full service restaurant in winter and summer weeks.}
  \label{fig.restaurantElectr}
\end{subfigure}
\\[0.2cm]
\begin{subfigure}{\textwidth}
   \begin{tikzpicture}
    \begin{axis}[width=\textwidth, height=0.175\textheight, xmin=0, xmax=168.1, ymax=1500, grid=major, xlabel=Hour, ylabel=kWh, xtick={0,24,48,72,96,120,144,168},legend style={at={(0.5,1.0)}, anchor=south,legend columns=-1}]
    \addplot[color=blue, dashed, line width = 2pt] table [x = Time, y = ElectrWinter] {HospitalWeek.tab};
    \addplot[color=black, line width = 2pt] table [x = Time, y = ElectrSummer] {HospitalWeek.tab};    
    \end{axis}
  \end{tikzpicture}
  \caption{Electricity demand for a hospital in winter and summer weeks}
  \label{fig.hospitalElectr}
\end{subfigure}
\\[0.2cm]
\begin{subfigure}{\textwidth}
  \begin{tikzpicture}
    \begin{axis}[width=\textwidth, height=0.175\textheight, xmin=0, xmax=168.1, ymax=115, grid=major, xlabel=Hour, ylabel=kWh, xtick={0,24,48,72,96,120,144,168},legend style={at={(0.5,1.0)}, anchor=south,legend columns=-1}]
    \addplot[color=blue, dashed, line width = 2pt] table [x = Time, y = ElectrWinter] {RetailWeek.tab};
    \addplot[color=black, line width = 2pt] table [x = Time, y = ElectrSummer] {RetailWeek.tab};    
    \end{axis}
  \end{tikzpicture}
  \caption{Electricity demand for a stand-alone retail in winter and summer weeks.}
  \label{fig.retailElectr}
\end{subfigure}
\caption{Electricity demand for five buildings in winter and summer weeks.}
\label{fig.Electr}
\end{figure}

\subsection{Numerical Experiment Setup} 

Numerical experiments are performed on a workstation with 32 GB memory and two Intel E5430 Xeon 4-core 2.66 GHz CPUs. We implement our algorithms in AMPL~\citep{fougayker12} to take advantage of AMPL's compact modeling syntax in profile generation, storage, and management. We use CPLEX version 12.6.1.0 as the MIP solver in AMPL. We set a 3-hour time limit and $1\%$ relative gap as the stopping criteria for CPLEX.

\subsection{Solutions of the Full Model with Short Second-Stage Horizons}
\label{sec.fullsol}

Table~\ref{tab.problemsize} shows the problem size of the full MILP model~\eqref{eqn.model} as a function of the number of days in the second stage. The number of binary variables, continuous variables, and constraints grows linearly with the number of days in the second-stage problem. For the one-year model, there are $1.05 \times 10^5$ binary variables, $2.71 \times 10^5$ continuous variables, and $6.11 \times 10^5$ constraints.

\begin{table}
  \centering
    \pgfplotstabletypeset[every head row/.style={before row=\hline,after row=\hline\hline}, every last row/.style={after row=\hline}, every first column/.style={ column type/.add={|}{} }, every last column/.style={ column type/.add={}{|} }] {FullModelExSize.tab}    
    \caption{Problem size of the MILP model~\eqref{eqn.model}. The number of binary variables, continuous variables, and constraints grows linearly with the number of days in the second stage.}
  \label{tab.problemsize}
\end{table}

\begin{table}
  \centering
  \begin{tabular}{|ccccccccc|}
    \hline
   Days   &   Time(s)  &  Nodes &  LP-iter  &    Bat &  Boil &  Chp &  Pow &  Stor \\
   \hline
    4     &   2      &  0   &    4.00e+03 &   1  & 1 &  2 &  0 &  6    \\
    7     &   241    &  2807&    4.49e+05 &   2  & 1 &  0 &  1 &  0     \\
   14     &   1.21e+03 & 10352&    2.73e+06 &   2  & 1 &  0 &  1 &  0     \\
   28     &   5.69e+03 & 47770&    6.02e+06 &   2  & 1 &  0 &  1 &  0     \\
84$^\dagger$ &  1.06e+04 &   487&    1.31e+06 &   6  & 1 &  0 &  1 &  0    \\
   \hline
  \end{tabular}
  \caption{Solutions of the MILP model~\eqref{eqn.model} for the restaurant example using CPLEX. The time in seconds, the number of nodes, and the number of simplex iterations grow exponentially with the horizon length in the second stage. The first-stage solutions for batteries (Bat), boilers (Boil), CHP-SOFC (CHP), Power SOFC (Pow), and water tank storage (Stor) vary with the problem size. $^\dagger$The 84-day model reaches the 3-hour limit, but the relative MIP gap is greater than $18\%$.}
  \label{tab.problemsol}
\end{table}

Table~\ref{tab.problemsol} shows the solutions of small problems for the restaurant example using CPLEX. We see the exponential increase of the amount of time, the number of nodes in the branch-and-bound trees, and the number of simplex iterations with the number of days. Solving the 84-day example is beyond the capabilities of CPLEX on the designated workstation. After reaching the 3-hour time limit, the relative MIP gap for the 84-day model is still greater than $18\%$.\footnote{Removing the time limit does not help because then the branch-and-bound tree generated by CPLEX will consume all allowable disk space of 100 GB on the workstation.} Our experience indicates that it is unlikely that we can solve our MILP over a ten-year time horizon.

Similar observations can be made for other building examples. In Figure~\ref{fig.fullmodelsol}, we see that for all five buildings, both the number of simplex iterations and the amount of solution time increase exponentially with the number of days in the second stage. 

\begin{figure}
  \centering
  \begin{tikzpicture}
    \begin{loglogaxis}[width=0.45\textwidth, xmin=0, xmax=84, grid=major, xlabel=Days, ylabel=Time (sec), xtick={4,7,14,28,84,364},log basis x=2,xticklabel=\pgfmathparse{2^\tick}\pgfmathprintnumber{\pgfmathresult},
legend columns=-1,
legend entries={\tc{black}{Office}, \tc{black}{SuperMarket}, \tc{black}{Restaurant}, \tc{black}{Retail}, \tc{black}{Hospital}},
legend to name=named]
    \addplot table [x = Days, y = Time] {FullModelSchool.tab};
    \addplot table [x = Days, y = Time] {FullModelSuperMarket.tab};
    \addplot table [x = Days, y = Time] {FullModelRestaurant.tab};
    \addplot table [x = Days, y = Time] {FullModelRetail.tab};
    \addplot table [x = Days, y = Time] {FullModelHospital.tab};
    \end{loglogaxis}
  \end{tikzpicture}  
  ~~
  \begin{tikzpicture}
    \begin{loglogaxis}[width=0.45\textwidth, xmin=0, xmax=84, grid=major, xlabel=Days, ylabel=LP Iterations, xtick={4,7,14,28,84,364},log basis x=2,xticklabel=\pgfmathparse{2^\tick}\pgfmathprintnumber{\pgfmathresult}]
    \addplot table [x = Days, y = LP-iter] {FullModelSchool.tab};
    \addplot table [x = Days, y = LP-iter] {FullModelSuperMarket.tab};
    \addplot table [x = Days, y = LP-iter] {FullModelRestaurant.tab};
    \addplot table [x = Days, y = LP-iter] {FullModelRetail.tab};
    \addplot table [x = Days, y = LP-iter] {FullModelHospital.tab};
    \end{loglogaxis}
  \end{tikzpicture} 
  \\
  \ref{named}
  \caption{Exponential increase of the amount of time and the number of simplex iterations with the number of days in the second stage of the full MILP model~\eqref{eqn.model}. The number of simplex iterations drops for the 84-day {\em restaurant\/} example because CPLEX reaches 3-hour time limit.}
  \label{fig.fullmodelsol}
\end{figure}
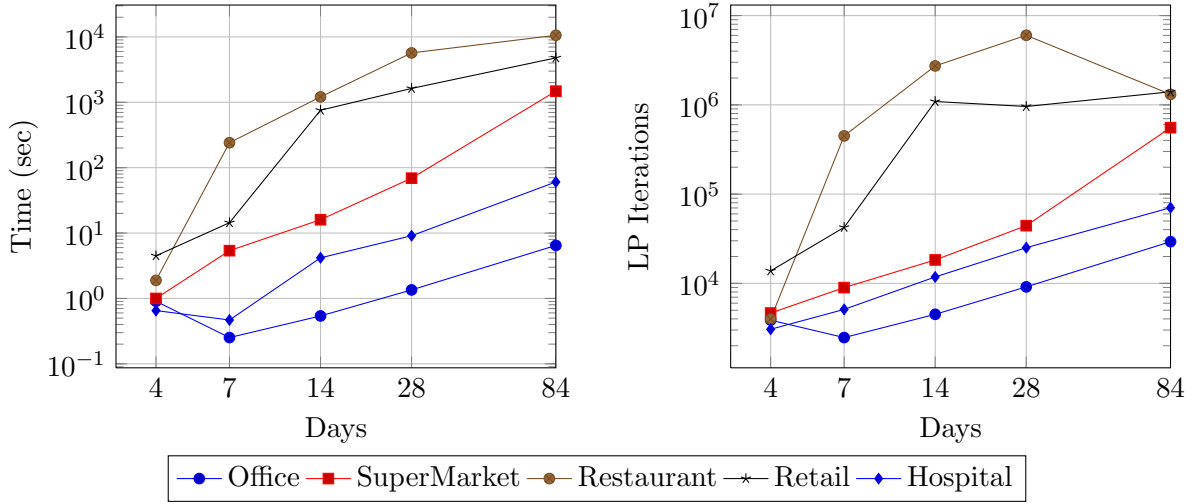

Additional numerical results are summarized in Table~\ref{tab.fullmodel} in Appendix \ref{sec.tables}. We point out that different buildings show different levels of difficulty for the MILP model~\eqref{eqn.model}. As we see in Figure~\ref{fig.fullmodelsol}, the five buildings differ by orders of magnitude in solution time and the number of simplex iterations. Similarly, the number of nodes in the branch-and-bound trees varies significantly over different buildings. For example, the 84-day model requires $10,384$ nodes for the retail example as opposed to $200$ nodes for the supermarket example; see Table~\ref{tab.fullmodel} in Appendix~\ref{sec.tables}.

\subsection{Solutions of  Semi-Coarse Model with Long Second-Stage Horizons}

In Figure~\ref{fig.probsizefullsemi}, we compare the problem size for the full model~\eqref{eqn.model}, the semi-coarse model~\eqref{eqn.semicoarse}, and the coarse model~\eqref{eqn.coarse} over a 10-year second-stage horizon. We see a roughly $8$ times reduction in the number of binary variables and $12$ times reduction in the number of continuous variables. The reason is that for each day with $24$ hours, we pick the three most frequent on/off profiles for binary variables and two cluster centers from the $k$-means clustering algorithm for continuous variables. While increasing the number of profiles improves the quality of coarsened models~\eqref{eqn.semicoarse} and~\eqref{eqn.coarse}, the resulting computational effort increases significantly. In practice, we find that the two-level approach strikes a good balance between solution quality and computational time with a small number of profiles (typically 2-3). Note that the large reduction in the number of constraints from the full model to the semi-coarse model. This is mainly because we have embedded the min/max production constraints~\eqref{eqn.powergen} and the switching constraints~\eqref{eqn.switching} in the profile generation; see Section~\ref{sec.semicoarse}. In particular, for a 10-year model with $12$ fuel cell units, there is a reduction of $12\times 87600 \times 4 \approx 4.2 \times 10^6$ constraints (that is, a $67\%$ of reduction in the number of constraints.)  

Table~\ref{tab.probsolsemi} shows the solution information for the semi-coarse model over a 10-year horizon. As opposed to the exponential increase for the full model, the solution time, the number of nodes, and the number of simplex iterations for the semi-coarse model grow more slowly. Furthermore, the first-stage solutions for the semi-coarse model are no longer sensitive to the horizon length. This result is in contrast to the variability of first-stage solutions for the full model in Table~\ref{tab.problemsol}.

Figure~\ref{fig.semicoarsesol} shows that both the number of simplex iterations and the amount of solution time scale more slowly for the semi-coarse models. For example, the number of simplex iterations for 10-year semi-coarse models is on the same order as that of the 84-day full model. The numerical results for the semi-coarse model are summarized in Table~\ref{tab.semicoarsemodel}  in Appendix~\ref{sec.tables}.

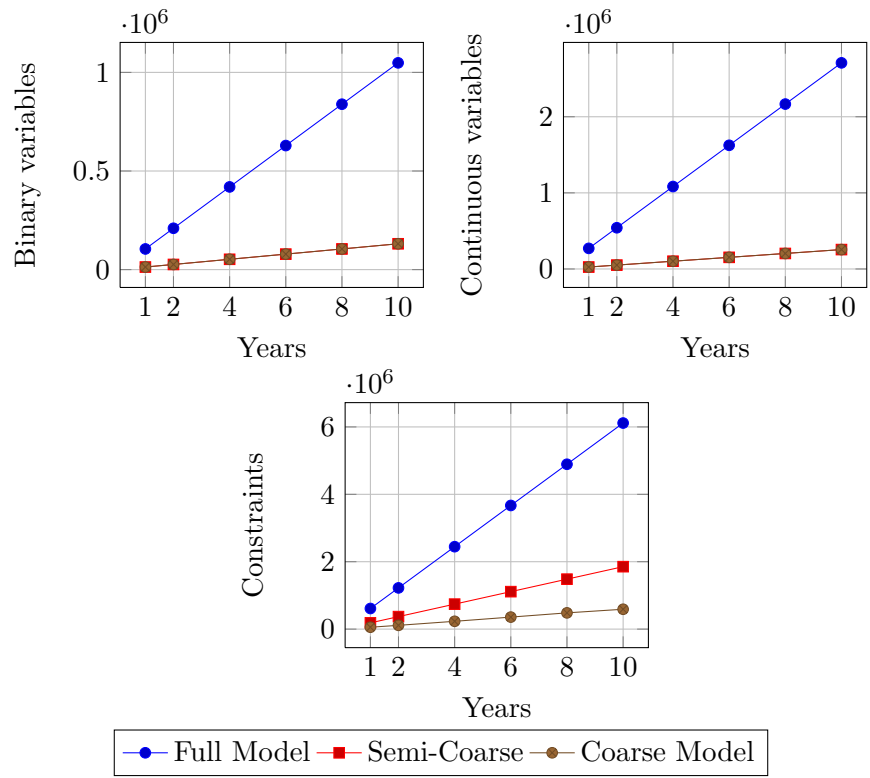
\begin{figure}
  \centering
    \begin{tikzpicture}
    \begin{axis}[width=0.34\textwidth, grid=major, xtick={1,2,4,6,8,10}, xlabel={Years}, ylabel=Binary variables,
legend columns=-1,
legend entries={\tc{black}{Full Model}, \tc{black}{Semi-Coarse}, \tc{black}{Coarse Model}},
legend to name=legend]
    \addplot table [x = Years, y = Binary] {FullModelExSizeLong.tab};
    \addplot table [x = Years, y = Binary] {SemiCoarseExSize.tab};
    \addplot table [x = Years, y = Binary] {CoarseModelExSize.tab};
    \end{axis}
    \end{tikzpicture}          
    \;
    \begin{tikzpicture}
    \begin{axis}[width=0.34\textwidth, grid=major, xtick={1,2,4,6,8,10}, xlabel={Years}, ylabel=Continuous variables]
    \addplot table [x = Years, y = Continuous] {FullModelExSizeLong.tab};
    \addplot table [x = Years, y = Continuous] {SemiCoarseExSize.tab};
    \addplot table [x = Years, y = Continuous] {CoarseModelExSize.tab};
    \end{axis}
    \end{tikzpicture}          
    \;
    \begin{tikzpicture}
    \begin{axis}[width=0.34\textwidth, grid=major, xtick={1,2,4,6,8,10}, xlabel={Years}, ylabel=Constraints]
    \addplot table [x = Years, y = Constraint] {FullModelExSizeLong.tab};
    \addplot table [x = Years, y = Constraint] {SemiCoarseExSize.tab};
    \addplot table [x = Years, y = Constraint] {CoarseModelExSize.tab};
    \end{axis}
    \end{tikzpicture}          
    \\
    \ref{legend}
  \caption{Problem size for the full MILP model~\eqref{eqn.model}, the semi-coarse model~\eqref{eqn.semicoarse}, and the coarse model~\eqref{eqn.coarse}. The number of binary and continuous variables are the same for the semi-coarse and coarse models.}
  \label{fig.probsizefullsemi}
\end{figure}

\begin{table}
  \centering
  \pgfplotstabletypeset[every head row/.style={before row=\hline,after row=\hline\hline}, every last row/.style={after row=\hline}, every first column/.style={ column type/.add={|}{} }, every last column/.style={ column type/.add={}{|} },every even row/.style={ before row={\rowcolor[gray]{0.9}}}] {SemiCoarseExSol.tab}    
  \caption{Solutions of the semi-coarse model~\eqref{eqn.semicoarse} for the restaurant example using CPLEX. The time in seconds, the number of nodes, and the number of simplex iterations grow much more slowly than for the full model in Table~\ref{tab.problemsol}. The first-stage solutions for batteries (Bat), boilers (Boil), CHP-SOFC (CHP), power SOFC (Pow), and water tank storage (Stor) do not change with the problem size.}
  \label{tab.probsolsemi}
\end{table}

\begin{figure}
  \centering
  \begin{tikzpicture}
    \begin{loglogaxis}[width=0.45\textwidth, xmin=0, xmax=10, grid=major, xlabel=Years, ylabel=Time (sec), 
 xtick={1,2,4,6,8,10},
log basis x=2,xticklabel=\pgfmathparse{2^\tick}\pgfmathprintnumber{\pgfmathresult},
legend columns=-1,
legend entries={\tc{black}{Office}, \tc{black}{SuperMarket}, \tc{black}{Restaurant}, \tc{black}{Retail}, \tc{black}{Hospital}},
legend to name=named]
    \addplot table [x = Years, y = Time] {SemiCoarseKmeansSchool.tab};
    \addplot table [x = Years, y = Time] {SemiCoarseKmeansSuperMarket.tab};
    \addplot table [x = Years, y = Time] {SemiCoarseKmeansRestaurant.tab};
    \addplot table [x = Years, y = Time] {SemiCoarseKmeansRetail.tab};
    \addplot table [x = Years, y = Time] {SemiCoarseKmeansHospital.tab};
    \end{loglogaxis}
  \end{tikzpicture}  
  ~~
  \begin{tikzpicture}
    \begin{loglogaxis}[width=0.45\textwidth, xmin=0, xmax=10, ymin=10000, ymax=1000000, grid=major, xlabel=Years, ylabel=LP Iterations,  xtick={1,2,4,6,8,10}, log basis x=2,xticklabel=\pgfmathparse{2^\tick}\pgfmathprintnumber{\pgfmathresult}]
    \addplot table [x = Years, y = LP-iter] {SemiCoarseKmeansSchool.tab};
    \addplot table [x = Years, y = LP-iter] {SemiCoarseKmeansSuperMarket.tab};
    \addplot table [x = Years, y = LP-iter] {SemiCoarseKmeansRestaurant.tab};
    \addplot table [x = Years, y = LP-iter] {SemiCoarseKmeansRetail.tab};
    \addplot table [x = Years, y = LP-iter] {SemiCoarseKmeansHospital.tab};
    \end{loglogaxis}
  \end{tikzpicture} 
  \\
  \ref{named}
  \caption{Solution time and number of simplex iterations for the semi-coarse model over a 10-year horizon. The growth of computational effort is considerably slower than that of the full model shown in Figure~\ref{fig.fullmodelsol}.}
  \label{fig.semicoarsesol}
\end{figure}
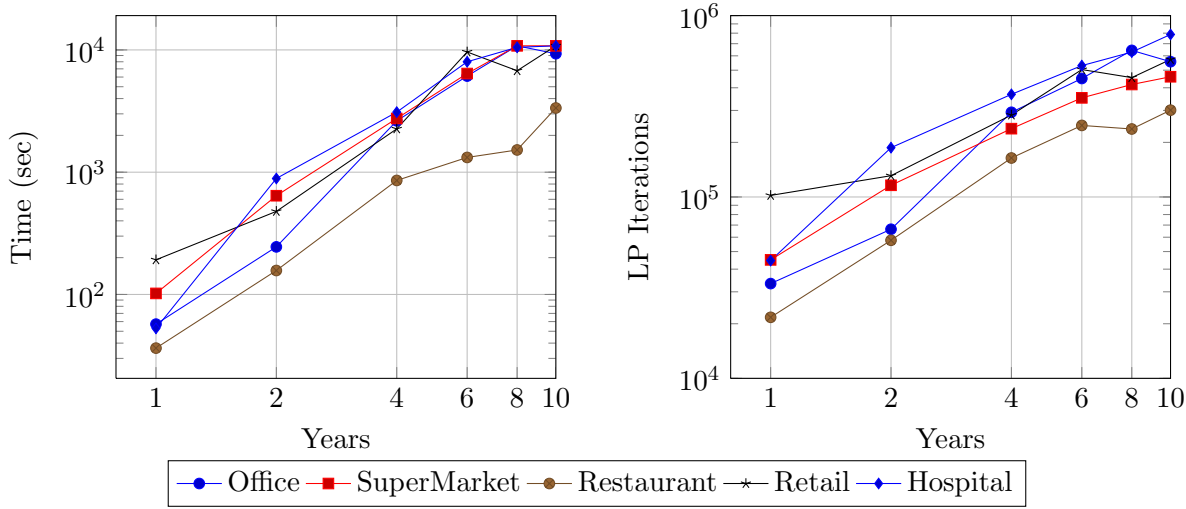

\subsection{Solutions of the Coarse Model with Long Second-Stage Horizons}
\label{sec.coarsesol}

As shown in Figure~\ref{fig.probsizefullsemi}, the semi-coarse and coarse models have the same number of binary and continuous variables, and the semi-coarse model has about twice as many constraints than does the coarse model. Therefore, one expects that both models can be solved with a similar amount of computational effort in terms of time and simplex iterations. Since we solve a sequence of coarse models, the total computational effort often exceeds that of solving the semi-coarse model. Because of the LP warm-start phase in Algorithm~\ref{alg.coarseMILP}, however, we find that the first iterate of the coarse MILPs provides remarkably good solutions of the semi-coarse model.  

To illustrate this point, we show in Table~\ref{tab.Restaurant10YearEx} the solution history of the MILP phase in Algorithm~\ref{alg.coarseMILP}. The MILP phase converges in three coarse MILP iterates, and the first iterate requires much more computational effort than do other iterates; in particular, it accounts for $76\%$ of total computational time. Compared with the solution from the semi-coarse model, the first iterate of the coarse MILPs provides a very good solution. The relative gap of the objective value 
\begin{equation}
  \label{eq.gap}
 \mu = ({\rm ObjVal}_{\rm semi} - {\rm ObjVal}_{\rm coar})/{\rm ObjVal}_{\rm semi}
\end{equation}
between the semi-coarse model and the coarse model at the first iterate is less than $4.7 \times 10^{-4}$. 

Figure~\ref{fig.OptGapTime} shows that the relative gap $\mu$ for all five building examples over the 10-year horizon is no greater than $0.05$. This implies that the first iterate of coarse MILPs is a good approximation of the semi-coarse model. Numerical results for the first iterate of coarse MILPs are summarized in Table~\ref{tab.coarsemodel} in Appendix~\ref{sec.tables}.

\begin{table}
  \centering
  \begin{tabular}{|ccccc|}
    \hline
    Iter   &       Binary   &   Continuous  & Constraint & Objective Value\\
\hline
 1     &       131,040   &     254,935    &    589,321  &  3,797,389.08\\
 2     &       131,040   &     254,935    &    589,617  & 3,797,565.30\\
 3     &       131,040   &     254,935    &    589,629  & 3,797,568.94\\

    \hline
  \end{tabular}
 \\[0.5cm]
 \pgfplotstabletypeset[every head row/.style={before row=\hline,after row=\hline\hline}, every last row/.style={after row=\hline}, every first column/.style={ column type/.add={|}{} }, every last column/.style={ column type/.add={}{|} },every even row/.style={ before row={\rowcolor[gray]{0.9}}}] {Restaurant10YearSol.tab}
  \caption{Problem size of MILP iterates for a 10-year model of the restaurant example. The first iterate accounts for $76\%$ of computational time, and the relative gap between the objective value defined in~\eqref{eq.gap} is less than $4.7 \times 10^{-4}$.}
  \label{tab.Restaurant10YearEx}
\end{table}

\begin{figure}
  \centering
  \begin{tikzpicture}
    \begin{semilogyaxis}[width=0.45\textwidth, xmin=0, xmax=10, grid=major, xlabel=Years, ylabel=$\mu$, 
 xtick={1,2,4,6,8,10}]
    \addplot table [x = Year, y expr=abs(\thisrow{SemiObj}-\thisrow{ObjVal})/\thisrow{SemiObj}] {CoarseModel1stIterSchool.tab};
    \addplot table [x = Year, y expr=abs(\thisrow{SemiObj}-\thisrow{ObjVal})/\thisrow{SemiObj}] {CoarseModel1stIterSuperMarket.tab};
    \addplot table [x = Year, y expr=abs(\thisrow{SemiObj}-\thisrow{ObjVal})/\thisrow{SemiObj}] {CoarseModel1stIterRestaurant.tab};
    \addplot table [x = Year, y expr=abs(\thisrow{SemiObj}-\thisrow{ObjVal})/\thisrow{SemiObj}] {CoarseModel1stIterRetail.tab};
    \addplot table [x = Year, y expr=abs(\thisrow{SemiObj}-\thisrow{ObjVal})/\thisrow{SemiObj}] {CoarseModel1stIterHospital.tab};
    \end{semilogyaxis}
  \end{tikzpicture}  
  \ref{named}
  \caption{The relative gap $\mu$ defined in~\eqref{eq.gap} between the first iterate of coarse MILPs and the semi-coarse model is no greater than 0.05 for all five buildings over a 10-year horizon.}
  \label{fig.OptGapTime}
\end{figure}
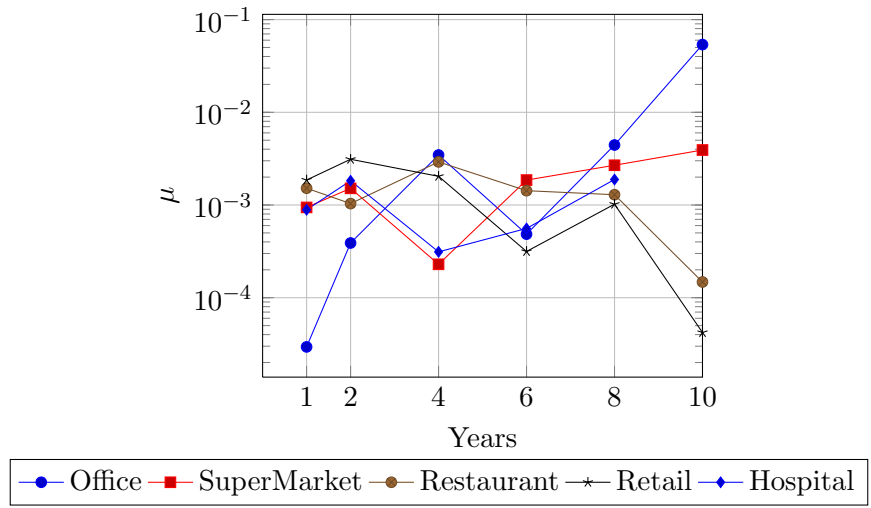

\subsection{Effect of the LP Warm-Start}

To illustrate the effect of the LP warm-start, we consider the cogeneration problem for the restaurant example over a one-year second-stage horizon. The original MILP~\eqref{eqn.model} has $1.05 \times 10^5$ binary variables, $2.71 \times 10^5$ continuous variables, and $6.11 \times 10^5$ constraints. Figure~\ref{fig.RestaurantWarmStart} shows the number of constraints and the computation time during the LP warm-start phase. A large number of constraints are added in the first few (3-5) iterates, resulting in more computation time at the beginning of the algorithm. As the LP warm-start phase progresses, fewer constraints are violated, and each LP-relaxation becomes easier to solve because of the warm-start. The amount of time in the LP warm-start phase is a fraction of that with one MILP re-solve in the MILP phase.
 
We next compare the MILP phase of Algorithm~\ref{alg.coarseMILP} with and without the LP warm-start. Figure~\ref{fig.RestaurantWarmStartMILP} shows the number of MILP re-solves and the computation time for both cases. Without the LP warm-start, Algorithm~\ref{alg.coarseMILP} takes 6 MILP re-solves, as opposed to 2 MILP re-solves with the LP warm-start. Moreover, the MILP re-solves with the LP warm-start is up to 20 times faster than the MILPs without the LP warm-start. We note that the improvement of the LP warm-start phase is observed in all building examples.

\begin{figure}
  \centering
  \begin{tabular}{cc}
    \begin{tikzpicture}
    \begin{axis}[width=0.45\textwidth, grid=major, xmin=1, xmax=20, xlabel={LP warm-start iterates}, ylabel={Constraints}]
    \addplot[color=blue, mark=*,line width = 1pt] table [x = Iter, y = Constraint] {CoarseModelLPRestaurant.tab};
    \end{axis}
    \end{tikzpicture}        
    &
    \begin{tikzpicture}
    \begin{axis}[width=0.45\textwidth, grid=major, xmin=1, xmax=20, xlabel={LP warm-start iterates}, ylabel={Time (sec)}]
    \addplot[color=blue, mark=*,line width = 1pt] table [x = Iter, y = Time] {CoarseModelLPRestaurant.tab};
    \end{axis}
    \end{tikzpicture}
    \end{tabular}   
  \caption{LP warm-start phase for the one-year restaurant example. The number of constraints (left panel) and the solution time (right panel) for LP-relaxations of~\eqref{eqn.coarse}.}
  \label{fig.RestaurantWarmStart}
\end{figure}
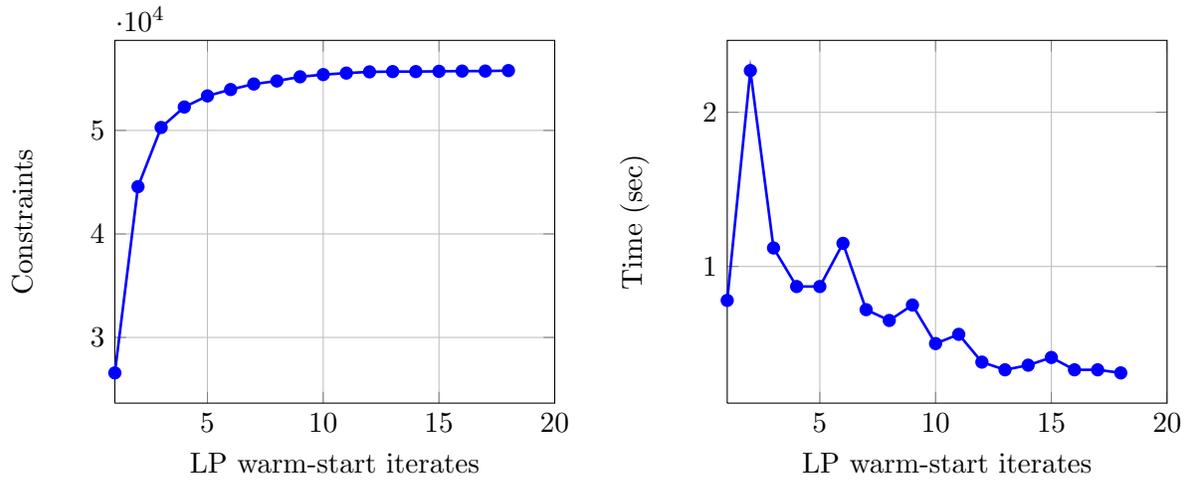
  
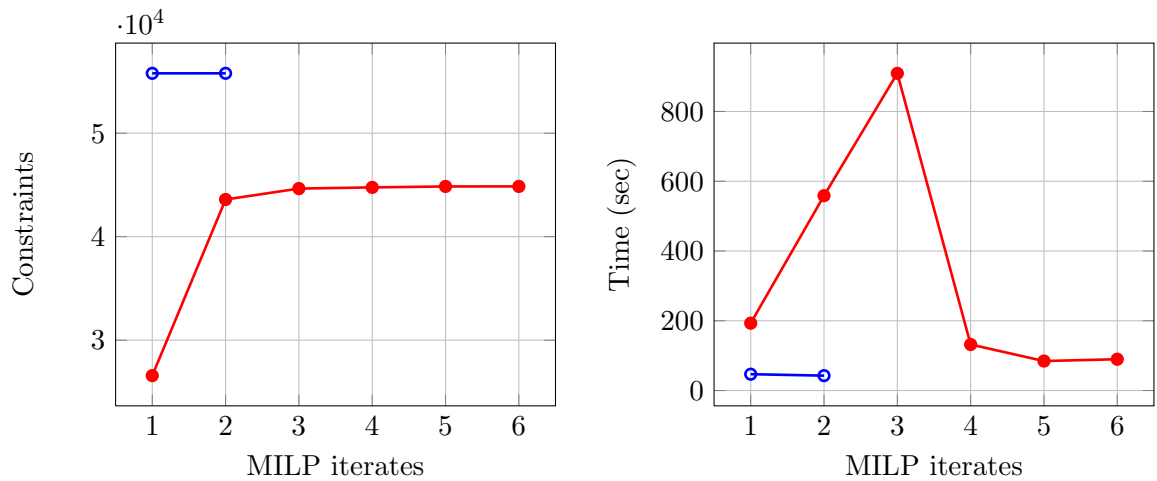
\begin{figure}
  \centering
    \begin{tabular}{cc}
    \begin{tikzpicture}
    \begin{axis}[width=0.45\textwidth, grid=major, xtick={1,2,3,4,5,6}, xlabel={MILP iterates}, ylabel={Constraints}]
    \addplot[color=red, mark=*,line width = 1pt] table [x = Iter, y = Constraint] {CoarseModelRestaurant_NoWarmStart.tab};
    \addplot[color=blue, mark=o,line width = 1pt] table [x = Iter, y = Constraint] {CoarseModelRestaurant.tab};
    \end{axis}
    \end{tikzpicture}        
    \quad &
    \begin{tikzpicture}
    \begin{axis}[width=0.45\textwidth, grid=major, xtick={1,2,3,4,5,6}, xlabel={MILP iterates}, ylabel={Time (sec)}]
    \addplot[color=red, mark=*,line width = 1pt] table [x = Iter, y = Time] {CoarseModelRestaurant_NoWarmStart.tab};
    \addplot[color=blue, mark=o,line width = 1pt] table [x = Iter, y = Time] {CoarseModelRestaurant.tab};
    \end{axis}
    \end{tikzpicture}
    \end{tabular}  
  \caption{Effect of LP warm-start for the one-year restaurant example. Algorithm~\ref{alg.coarseMILP} with LP warm-start phase (\tc{blue}{$\circ$}) converges with fewer MILP iterations than without LP warm-start (\tc{red}{$\bullet$}). While more constraints are added to MILPs with LP warm-start (left panel), each MILP iterate is 5 to 20 times faster than MILPs without LP warm-start (right panel).}
  \label{fig.RestaurantWarmStartMILP}
\end{figure}

\subsection{First-Stage Solutions}

Since the first-stage solutions are of primary importance, we next compare them from the full MILP, the semi-coarse, and the coarse models.  In Table~\ref{tab.fullmodel}, we see that the investment decisions at the first-stage vary significantly with respect to the horizon length of the second-stage problem. For example, the number of water tanks for storage drops from six units in the 4-day model to zero units in the 7-day model, and the number of batteries increases from two units in the 28-day model to six units in the 84-day model. These numbers imply that the optimal first-stage solutions from a short-horizon problem are suboptimal solutions for a long-horizon problem. Therefore, we must solve long-horizon MILPs~\eqref{eqn.model} as done for the semi-coarse and coarse models. 

In Table~\ref{tab.semicoarsemodel}, the first-stage solutions are much less sensitive to the horizon length at the second stage. For example, the first-stage solutions do not change over the 10-year horizon for the supermarket and restaurant examples. For the other three buildings, the $\ell_1$-norm between first-stage solutions of the same building example is no greater than 2 (i.e., at most two-unit difference of investment decisions). Similarly, the first-stage solutions at the first iterate of the coarse MILPs are insensitive to the horizon length at the second stage. The $\ell_1$-distance between the first-stage solutions of the same building example is no greater than 3. 


\section{Conclusions} \label{sec.concl}

We study two-stage MILPs with more than 1 million binary variables in the second-stage problem. We develop a two-level approach by constructing semi-coarse models (coarsened with respect to variables) and coarse models (coarsened in both variables and constraints). We show that the semi-coarse model is guaranteed to provide a feasible solution of the original MILP and hence results in an upper bound on the optimal solution. We solve a sequence of coarse MILPs with aggregated constraints that converges to the same upper bound with a finite number of steps. Furthermore, we take advantage of the LP warm-start to reduce the number of MILP re-solves.

We apply our approach to the cogeneration problem for commercial buildings. We demonstrate the effectiveness of the two-level approach using building examples with simulation data. In particular, the two-level approach allows us to obtain good approximate solutions at a fraction of the time required for solving the original MILPs. Furthermore, we show that the two-level approach scales to large problems that are beyond the capacity of state-of-the-art commercial solvers.

A number of extensions are of interest in our future work. First, how should one add new profiles in the coarsened models? Currently, we select a number of prespecified profiles and fix them throughout the MILP re-solves. Since the solution quality of coarsened models is determined by profiles, it is of interest to add new profiles dynamically that are most promising in reducing the objective value. It seems that the reduced cost associated with profiles can be obtained by solving appropriate pricing problems like those in the column generation approach~\citep{barjohnemsavvan98}. Further, it is also of interest to consider dynamic aggregations of constraints as new profiles are included. For set-partitioning constraints, a similar approach has been developed in the column generation framework~\citep{elhvilsouguy05}. We plan to explore this line of research in future work.

Second, how can one obtain lower bounds on the optimal solution in the two-level framework? Since the coarse model is a relaxation of the semi-coarse model, which itself is a tightening of the full model, the coarse model is not a relaxation of the original MILP. While relaxing binary variables in the original MILP results in a lower bound, our numerical results indicate a sizable gap between this lower bound and the optimal solution. It is an open problem how to generate lower bounds using appropriate coarse MILPs.

Furthermore, advanced multilevel approaches for PDEs typically require several sweeps of fine/coarse levels in order to achieve accurate solutions~\citep{fal06}. It may be beneficial to iterate between solving coarse models to optimality and solving ``partially'' the fine model until a prespecified accuracy for the solution is achieved. We intend to extend our two-level framework in this direction.


\addcontentsline{toc}{section}{Acknowledgments}
\section*{Acknowledgments}

This material is based upon work supported by the U.S. Department of Energy, Office of Science, Office of Advanced Scientific Computing Research, Applied Mathematics program under contract number DE-AC02-06CH11357.

\appendix

\section{Full Model for the Cogeneration Problem}

For the sake of completeness, we describe here our cogeneration model.

\paragraph{Index Sets.} We use calligraphic upper-case letters for all sets.
The set of technologies: batteries, boilers, CHP-SOFC, power SOFC, and water tank storage is denoted by ${\cal J}$. We use batt, boil, pow, chp, and stor for shortcuts.
The subset ${\cal J}_g = \{ \mbox{power SOFC, CHP-SOFC} \}$ denotes SOFC generation technologies that require on/off operations at the second stage. Each technology $j \in \cJ_g$ has a number of identical units, and the index set is denoted by ${\cal U}_j$. Finally, $\cT$ is the index set of time, ${\cal M}$ is the index set of months, and ${\cal T}_m \subset \cT$ is the index set of time for each month $m \in {\cal M}$. 

\paragraph{Variables.} We use lower-case letters to denote variables. The first-stage variables in our model are the number of units of technology $j \in {\cal J}$, which we denote by $y_j \in \Z_+$.
All other variables model operation of the installed units, and they are our second-stage variables. We use the subscript $t$ to indicate the time period:
  $b_t$ and $b_t^{\rm IO}$  are the power storage and input/output for batteries, respectively;  
  $s_t$, $s_t^{\rm in}$, and $s_t^{\rm out}$  are the heat storage, input, and output for water tanks, respectively;  
  $p_{jt}$ and $q_{jt} \geq 0$ are the power and heat output of technology $j$, respectively;   
  $u_t \geq 0$             is the power purchased from the utility and
  $u_m^{\rm max} \geq 0$    is the maximum power purchased in month $m$;
  $x_{ijt} \in \{0,1\}$    indicates whether unit $i$ of technology $j$ operates in time period $t$;
  and $w_{ijt} \in [0,1]$      is an auxiliary variable to indicate whether unit $i$ of technology $j$ switches from $t$ to $t+1$. 
Note that we do not impose binary constraint on $w_{ijt}$ because it takes a binary value at the solution due to the switching constraint~\eqref{eqn.switching} and the binary constraint $x_{ijt} \in \{0,1\}$.
We use the convention $t \in \cT$, $j \in \cJ$, $i \in \cU_j$, $m \in {\cal M}$ unless explicitly mentioned.

\paragraph{Parameters.} We use upper-case letters to denote parameters and constants:
  $C_j$            is the capital and installation cost of technology $j$;  
  $H$              is the number of hours in the lifetime of technology (e.g., $H = 87600$ for a ten-year model);
  $T$              is the number of hours in the problem horizon;
  $Y$              is the hourly discount factor with $3\%$ annual interest rate (i.e., $Y = 1 - 0.03/8760$);
  $M_{j}$          is the operation and maintenance cost of technology $j$;
  $P_t$ and $G_t$     are the price for electricity and natural gas, respectively;
  $P_m^{\rm max}$  is the peak demand price for power from the utility in month $m$; 
  $W_j$            is the cost of switching a unit of technology $j$ on or off; 
  $S^c_j$          is the thermal capacity per unit for technology $j$; 
  $B^c$            is the power capacity per unit for batteries;
  $D_t^P$ and $D_t^Q$  are the power and heat demand in period $t$, respectively;
  $R^{\rm min}_j$ and $R^{\rm max}_j$  are the minimum and maximum power output when a unit of technology $j$ is turned on, respectively; 
  $E_j^P$ and  $E_j^Q$     are the power and thermal efficiency of technology $j$, respectively;   
  $L^P$ and $L^Q$  are the average power loss in battery and heat loss in water tank storage, respectively.

\paragraph{Cost Function.} We now formulate the complete MILP model for cogeneration problem in~\eqref{eqn.model}. The objective function consists of two parts: the capital and installation cost in the first-stage and the operation cost in the second stage. The operation cost includes the peak power usage, operation and maintenance, switching units, gas, and purchased power costs. We discount the second-stage costs with an {\em hourly\/} discount factor based on a $3\%$ {\em annual\/} interest rate; hence $Y=1-0.03/8760$. Therefore, the cost incurred in hour $t$ is multiplied by the discount factor $Y^t$.\footnote{For comparison, the annual discount factor is $Y^{8760} \approx 97.044\%$, which is consistent with the $3\%$ interest rate; and a ten-year discount factor is $Y^{87600} \approx 74.082\%$.} The peak demand charge at the end of month $m$ is discounted with $Y^{t_m}$. To compare cost over different time horizons $T$, we compute the average hourly cost and multiply the result with the lifetime of technologies, $H$, resulting in the factor $H/T$ in the second-stage cost in~\eqref{eqn.model}.

\bsub \label{eqn.model}
\begin{alignat}{2}
     \mini \quad 
           & \ds \sum_{j \in \cJ} C_j y_j  + \frac{H}{T} \Big\{ \sum_{m \in {\cal M}} Y^{t_m} P_m^{\rm max} u_m^{\rm max}  
           &&  \nonumber  \\
           & \qquad \qquad + \ds \sum_{t \in \cT} Y^t \big[ \sum_{j \in \cJ_g} (M_{j} p_{jt}  + \ds \sum_{i \in {\cal U}_j}  W_j w_{ijt})                                        + G_t q_{t} + P_t u_t \big] \Big\} 
           && \nonumber  \\
     \st \quad & y_j \in \Z, ~ 0 \leq y_j \leq |{\cal U}_j| 
               && \label{eqn.int} \\
               & \ds \sum_{j \in \cJ_g} p_{jt} + u_t - b_t^{\rm IO} \geq  D_t^P 
               &&    \label{eqn.PowerDemand} \\
               & \dps R_j^{\rm min}  x_{ijt}  \leq  p_{ijt}  \leq  R_j^{\rm max} x_{ijt}, \quad
                  p_{jt} = \dps \sum_{i \in {\cal U}_j} p_{ijt} 
               && j \in {\cal J}_g \label{eqn.powergen} \\
               &  x_{ijt} \in \{0,1\}, \quad \ds  \sum_{i \in {\cal U}_j} x_{ijt} \leq y_j  
               &&  j \in {\cal J}_g  \label{eqn.OnOff} \\
               &  x_{(i+1)jt} \leq x_{ijt}
               && j \in \cJ_g, i < |{\cal U}_j| \label{eqn.symbrk} \\
               &  x_{ij(t+1)} - x_{ijt} \leq w_{ijt}, \, x_{ijt} - x_{ij(t+1)} \leq w_{ijt}, \, 0 \leq w_{ijt} \leq 1, \; \;
               &&  j \in {\cal J}_g,  t < |\cT| \label{eqn.switching}\\
               & 0 \leq u_t \leq u_m^{\rm max}
               &&  t \in \cT_m  \label{eqn.maxpower} \\
               & b_{t+1} = (1 - L^P) b_{t} + b_t^{\rm IO} & &   t < |\cT| \label{eqn.battery} \\
               & b_1 \,=\, b_{|\cT|}   & & \label{eqn.batteryboundary}  \\                        
               & 0 \leq b_t \leq B^c y_{\rm batt}  
               &&       \label{eqn.batterybnd} \\
               & s^{\rm out}_{t} + q_{t} \geq D_t^Q & &   \label{eqn.HeatDemand} \\
               & s_{t+1} - (1 - L^Q) s_{t} + s^{\rm out}_{t} \leq (E_{\rm chp}^Q /E_{\rm chp}^P)   p_{{\rm chp},t} 
               && t < |\cT| \label{eqn.storage} \\
               &  s_{1} = s_{|\cT|} &&  \label{eqn.heatbnd} \\
               & s^{\rm out}_{t} \leq  s_{t} \leq  S^c_{\rm stor} y_{\rm stor}  && \label{eqn.storbnd} \\
               & q_{t} \leq  S^c_{\rm boil} y_{\rm boil} & &    \label{eqn.boilbnd} \\
               & b_t, p_{jt}, q_{t}, s_t, s_t^{\rm out}, u_t \geq 0 &&   \label{eqn.positive}
\end{alignat}
\esub

\paragraph{Constraints.} The constraints in the MILP model~\eqref{eqn.model} can be divided into three groups. The first group of constraints couples first- and second-stage variables; in particular, the number of on-units for fuel cells and the capacity of battery, storage, and boiler are bounded by the number of units purchased in the first-stage, as described in \eqref{eqn.OnOff}, \eqref{eqn.batterybnd}, \eqref{eqn.storbnd}, and \eqref{eqn.boilbnd}, respectively. The second group of constraints couples variables across technologies; in particular, technologies are coordinated to meet the power demand (\ref{eqn.PowerDemand}) and the heat demand (\ref{eqn.HeatDemand}). We note that dualizing these constraints results in decoupled variables for each technology; see, for example, ~\citet{fis81} and \citet{bar88}. However, MILP~\eqref{eqn.model} does not lend itself to this Lagrangian relaxation approach because of coupling constraints in the third group. This group of constraints couples variables over time periods in the second stage. In particular, the battery and storage constraints \eqref{eqn.battery} and \eqref{eqn.storage} link variables over the entire horizon. In addition, constraint (\ref{eqn.maxpower}), which models the maximum power purchased from the utility in each months, links variables over time instances in each month. Moreover, constraint~\eqref{eqn.switching}, which models on/off switches of the generation units, couples binary variables $x_{ijt}$ over two {\em immediate\/} time instances. Boundary conditions~(\ref{eqn.batteryboundary}) and (\ref{eqn.heatbnd}) for battery and storage couple variables at both ends of the entire horizon. 

We conclude this subsection by explaining the rest of the constraints in \eqref{eqn.model}. The number of purchased units is upper bounded in (\ref{eqn.int}). Since all units of the same technology are identical, we introduce a symmetry breaking constraint \eqref{eqn.symbrk} that states that if unit $i+1$ is on, then unit $i$ must be on as well. When unit $i$ of technology $j$ is turned on, the power generated $p_{ijt}$ is bounded by the minimum and maximum capacity in~\eqref{eqn.powergen}. The non-negativity constraint for variables is given by~\eqref{eqn.positive}. 

\section{Tables of Numerical Results}
\label{sec.tables}

\newpage
\newgeometry{margin=1.5cm} 
\begin{landscape}
\begin{table}
  \centering
\pgfplotstabletypeset[every head row/.style={before row=\hline,after row=\hline\hline}, every last row/.style={after row=\hline}, every first column/.style={ column type/.add={|}{} }, every last column/.style={ column type/.add={}{|} },every even row/.style={ before row={\rowcolor[gray]{0.9}}}] {FullModel.tab}
 \caption{Problem size and AMPL/CPLEX solution information for the cogeneration problem~(\ref{eqn.model}) for five buildings: 1-Office, 2-Supermarket, 3-Restaurant, 4-Retail, 5-Hospital. Model instances that reach 3-hour time limit are indicated by MSG=1.}
 \label{tab.fullmodel}
\end{table}
\end{landscape}
\restoregeometry

\newgeometry{margin=1.5cm} 
\begin{landscape}
\begin{table}
  \centering
  \pgfplotstabletypeset[every head row/.style={before row=\hline,after row=\hline\hline}, every last row/.style={after row=\hline}, every first column/.style={ column type/.add={|}{} }, every last column/.style={ column type/.add={}{|} },every even row/.style={ before row={\rowcolor[gray]{0.9}}}] {SemiCoarseModelKmeans.tab}
 \caption{Problem size and AMPL/CPLEX solution information for the semi-coarse model~(\ref{eqn.semicoarse}) for five buildings: 1-Office, 2-Supermarket, 3-Restaurant, 4-Retail, 5-Hospital. Model instances that reach 3-hour time limit are indicated by MSG=1.}
 \label{tab.semicoarsemodel}
\end{table}
\end{landscape}
\restoregeometry

\newgeometry{margin=1.5cm} 
\begin{landscape}
\begin{table}
  \centering
  \pgfplotstabletypeset[every head row/.style={before row=\hline,after row=\hline\hline}, every last row/.style={after row=\hline}, every first column/.style={ column type/.add={|}{} }, every last column/.style={ column type/.add={}{|} },every even row/.style={ before row={\rowcolor[gray]{0.9}}}] {CoarseModel1stIter.tab}
 \caption{Problem size and AMPL/CPLEX solution information for the first iteration of the coarse model~(\ref{eqn.semicoarse})  for five buildings: 1-Office, 2-Supermarket, 3-Restaurant, 4-Retail, 5-Hospital. Model instances that reach 3-hour time limit are indicated by MSG=1.}
 \label{tab.coarsemodel}
\end{table}
\end{landscape}
\restoregeometry


\section{Semi-Coarse Model for the Cogeneration Problem}

\bsub \label{eqn.semicoarse}
\begin{alignat}{2}
     \mini \quad 
           & \ds \sum_{j \in \cJ} C_j y_j 
           + \frac{H}{T} \Big\{ \sum_{m \in {\cal M}} Y^{t_m} P_m^{\rm max} u_m^{\rm max} +
                              \ds \sum_{d} \bar{Y}_d^T \big[ \sum_{j,i,k,l}  \bar{p}_{ijdkl} M_{j} \bar{P}_{jkl}
           &&  \nonumber \\
           &  \qquad  \qquad + \ds \sum_{d,j,i,k} \bar{x}_{ijdk} W_j \bar{W}_{jk} + \sum_{d, k \in \cP_q} \bar{q}_{dk} \bar{Q}_{k} + \sum_{d,k \in \cP_u} \bar{u}_{dk} \bar{U}_{k} \big] \Big\}
           && \nonumber \\  
     \st \quad 
           & y_j \in \mathbb{Z}, ~ 0 \leq y_j \leq |\cU_j| 
           && j \in \cJ  \\
           & \ds \sum_{j,i,k,l} \bar{p}_{ijdkl} \bar{P}_{jkl} + \sum_{k \in \cP_u} \bar{u}_{dk} \bar{U}_{k} - \sum_{k \in \cP_b} \bar{b}_{dk}  \bar{B}_{k}^{\rm IO} \geq  \bar{D}_d^P 
           && \label{eqn.semi-powerdemand} \\
           & \bar{x}_{ijdk} \in \{0,1\}, ~ \sum_{k \in \cP_o} \bar{x}_{ijdk} \leq 1 
           && \label{eqn.setppcx}  \\
           & \ds \sum_{i,k} \bar{x}_{ijdk} \bar{X}_{jk} \leq y_j {\bf 1} 
           && \label{eqn.semi-cap-sofc}  \\
           & \sum_{k \in \cP_o} \bar{x}_{(i+1)jdk} \bar{X}_{jk} \leq  \sum_{k \in \cP_o} \bar{x}_{ijdk} \bar{X}_{jk}
           && i < |\cU_j| \label{eqn.semi-symbrk} \\
           & \bar{p}_{ijdk} \in [0,1], ~ \sum_{l \in \cP_p} \bar{p}_{ijdkl} = \bar{x}_{ijdk}, ~  
             \sum_{k \in \cP_o} \sum_{l \in \cP_p} \bar{p}_{ijdkl} \leq  1
           && \label{eqn.setppcv}  \\ 
           & 0 \leq \sum_{k \in \cP_u} \bar{u}_{dk} \bar{U}_{k} \leq u_m^{\rm max} {\bf 1}
           &&  d \in \cD_m   \label{eqn.semi-maxdemand} \\
           & \sum_{k \in \cP_b} \bar{b}_{(d+1)k} \bar{B}_k(1) = \sum_{k \in \cP_b} \bar{b}_{dk} \big[ (1-L^P) \bar{B}_k(\delta)  + \bar{B}_k^{\rm IO}(\delta) \big]
           && d < |\cD| \label{eqn.semi-midnights} \\
           & \sum_{k \in \cP_b} \bar{b}_{1k} \bar{B}_{k}(1) \,=\, \sum_{k \in \cP_b} \bar{b}_{|\cD|k} \bar{B}_{k}(\delta)
           && \label{eqn.semi_batt}  \\                        
           & \sum_{k \in \cP_b} \bar{b}_{dk} \bar{B}_k \leq B^c y_{\rm batt} {\bf 1} 
           && \label{eqn.semi-cap-batt} \\
           & \ds \sum_{k \in \cP_s} \bar{s}_{dk} \bar{S}^{\rm out}_{k} + \sum_{k \in \cP_q} \bar{q}_{dk} \bar{Q}_{k} \geq \bar{D}_d^Q 
           && \label{eqn.semi-HeatDemand} \\
           & \sum_{k \in \cP_s} \bar{s}_{dk} \big[ S_\delta - (1-L^Q)I_\delta \big] \bar{S}_{k} + \sum_{k \in \cP_s} \bar{s}_{(d+1)k} E_\delta \bar{S}_{k} 
           &&  d < |\cD| \nonumber \\
           & + \sum_{k \in \cP_s} \bar{s}_{dk} \bar{S}^{\rm out}_{k}  \leq (E_{j}^Q /E_{j}^P) \sum_{i,k,l} \bar{p}_{ijdkl} \bar{P}_{jkl} 
           &&       j = {\rm chp}  \label{eqn.semi-storage} \\
           & \sum_{k \in \cP_s} \bar{s}_{1k} \bar{S}_{k}(1) \,=\, \sum_{k \in \cP_s} \bar{s}_{|\cD|k} \bar{S}_{k}(\delta)
           && \label{eqn.semi-heatbnd}  \\                        
           & \ds \sum_{k \in \cP_s} \bar{s}_{dk} \bar{S}^{\rm out}_{k} \leq  S^c_{\rm stor} y_{\rm stor} {\bf 1}  
           &&  \label{eqn.semi-cap-stor} \\
           & \ds \sum_{k \in \cP_q} \bar{q}_{dk} \bar{Q}_{k}  \leq S^c_{\rm boil} y_{\rm boil} {\bf 1}   
           &&  \label{eqn.semi-cap-boil} \\       
           & \bar{u}_{dk}, ~\bar{b}_{dk}, ~\bar{s}_{dk}, ~\bar{q}_{dk} \in [0,1]
           &&  \label{eqn.boundsw} \\
           &    \sum_{k \in \cP_u} \bar{u}_{dk} \leq 1,  
                \sum_{k \in \cP_b} \bar{b}_{dk} \leq 1, 
                \sum_{k \in \cP_s} \bar{s}_{dk} \leq 1, 
                \sum_{k \in \cP_q} \bar{q}_{dk} \leq 1.
           && \label{eqn.setppcw}
\end{alignat}
\esub


\section{Coarse Model for the Cogeneration Problem}

\bsub \label{eqn.coarse}
\begin{alignat}{2}
     \mini \quad 
           & \ds \sum_{j \in \cJ} C_j y_j 
           + \frac{H}{T} \Big\{ \sum_{m \in {\cal M}} Y^{t_m} P_m^{\rm max} u_m^{\rm max} +
                              \ds \sum_{d} \bar{Y}_d^T \big[ \sum_{j,i,k,l}  \bar{p}_{ijdkl} M_{j} \bar{P}_{jkl}
           &&  \nonumber \\
           &  \qquad  \qquad + \ds \sum_{d,j,i,k} \bar{x}_{ijdk} W_j \bar{W}_{jk}  + \sum_{d,k \in \cP_u} \bar{u}_{dk} \bar{U}_{k}  + \sum_{d, k \in \cP_q} \bar{q}_{dk} \bar{Q}_{k} \big] \Big\}
           && \nonumber \\  
     \st \quad & y_j \in \mathbb{Z}, ~ 0 \leq y_j \leq |\cU_j| 
               && j \in \cJ \\
               &  \ds \sum_{j,i,k,l} \bar{p}_{ijdkl} \hat{\bar{P}}_{jkl} + \sum_{k \in \cP_u} \bar{u}_{dk} \hat{\bar{U}}_{k}  - \sum_{k \in \cP_b} \bar{b}_{dk} \hat{\bar{B}}_{k}^{\rm IO} \geq  \hat{\bar{D}}_d^{P}  
               &&   \label{eqn.coarse-powerdemand}  \\
               & \bar{x}_{ijdk} \in \{0,1\}, ~ \sum_{k \in \cP_o} \bar{x}_{ijdk} \leq 1 
               && \label{eqn.coarse-setppcx} \\
               & \ds  \sum_{i,k} \bar{x}_{ijdk}  \hat{\bar{X}}_{jk} \leq \delta \cdot y_j 
               &&  \label{eqn.coarse-cap-sofc}  \\
               & \sum_{k \in \cP_o} \bar{x}_{ijdk} \hat{\bar{X}}_{jk} \geq \sum_{k \in \cP_o} \bar{x}_{(i+1)jdk} \hat{\bar{X}}_{jk} 
               && i< |\cU_j|  \label{eqn.coarse-symmbrk} \\
               & \bar{p}_{ijdk} \in [0,1], ~ \sum_{l \in \cP_p} \bar{p}_{ijdkl} = \bar{x}_{ijdk}, ~  
                 \sum_{k \in \cP_o} \sum_{l \in \cP_p} \bar{p}_{ijdkl} \leq 1
               && \label{eqn.coarse-setppcv} \\ 
               & 0 \leq \sum_{k \in \cP_u} \bar{u}_{dk}  \hat{\bar{U}}_{k}  \leq \delta \cdot u_m^{\rm max} 
               && d\in \cD_m \label{eqn.coarse-maxdemand} \\
               & \sum_{k \in \cP_b} \bar{b}_{(d+1)k} \bar{B}_k(1) = \sum_{k \in \cP_b} \bar{b}_{dk} ( (1-L^P) \bar{B}_k(\delta) + \bar{B}_k^{\rm IO}(\delta)) 
               && d < |\cD| \label{eqn.coarse-batt-midnight} \\
               & \sum_{k \in \cP_b} \bar{b}_{1k} \bar{B}_{k}(1) \,=\, \sum_{k \in \cP_b} \bar{b}_{|\cD|k} \bar{B}_{k}  (\delta) 
               && \label{eqn.coarse-batt-bnd} \\
               & \sum_{k \in \cP_b} \bar{b}_{dk}  \hat{\bar{B}}_k \leq \delta \cdot B^c y_{\rm batt} 
               && \label{eqn.coarse-cap-batt} \\
               &  \ds \sum_{k \in \cP_s} \bar{s}_{dk} \hat{\bar{S}}^{\rm out}_{k} + \sum_{k \in \cP_q} \bar{q}_{dk} \hat{\bar{Q}}_{k} \geq  \hat{\bar{D}}_d^Q 
               &&  \\
               & \sum_{k \in \cP_s} \bar{s}_{dk} (L^Q\hat{\bar{S}}_{k}-\bar{S}_k(1)) + \sum_{k \in \cP_s} \bar{s}_{(d+1)k} \bar{S}_{k}(1)
               &&  d < |\cD| \nonumber \\
               & + \sum_{k \in \cP_s} \bar{s}_{dk}  \hat{\bar{S}}^{\rm out}_{k}  \leq (E_{j}^Q /E_{j}^P) \sum_{i,k,l} \bar{p}_{ijdkl}  \hat{\bar{P}}_{jkl} 
               &&       j = {\rm chp}  \\
               & \sum_{k \in \cP_s} \bar{s}_{1k} \bar{S}_{k}(1) \,=\, \sum_{k \in \cP_s} \bar{s}_{|\cD|k} \bar{S}_{k} (\delta)  
               && \label{eqn.coarse-stor-bnd} \\
               & \ds \sum_{k \in \cP_s} \bar{s}_{dk}  \hat{\bar{S}}^{\rm out}_{k} \leq  \delta \cdot S^c_{\rm stor} y_{\rm stor}  
               && \label{eqn.coarse-cap-stor} \\
               & \ds \sum_{k \in \cP_q} \bar{q}_{dk} \hat{\bar{Q}}_{k}  \leq \delta \cdot S^c_{\rm boil} y_{\rm boil} 
               && \label{eqn.coarse-cap-boil}  \\       
               & \bar{u}_{dk}, ~\bar{b}_{dk}, ~\bar{s}_{dk}, ~\bar{q}_{dk} \in [0,1] 
               &&  \nn \\
               & \sum_{k \in \cP_u} \bar{u}_{dk} \leq 1,  
                \sum_{k \in \cP_b} \bar{b}_{dk} \leq 1, 
                \sum_{k \in \cP_s} \bar{s}_{dk} \leq 1, 
                \sum_{k \in \cP_q} \bar{q}_{dk} \leq 1.  
               && \label{eqn.coarse-setppcw}
\end{alignat}
\esub



\addcontentsline{toc}{section}{References}
 \bibliographystyle{apalike}
\bibliography{StochMIP}

\vfill
\begin{flushright}
\scriptsize
\framebox{\parbox{\textwidth}{The submitted manuscript has been created
by the UChicago Argonne, LLC, Operator of Argonne
National Laboratory (``Argonne'') under Contract No.\
DE-AC02-06CH11357 with the U.S. Department of Energy.
The U.S. Government retains for itself, and others
acting on its behalf, a paid-up, nonexclusive, irrevocable
worldwide license in said article to reproduce,
prepare derivative works, distribute copies to the
public, and perform publicly and display publicly, by or on
behalf of the Government.}}
\normalsize
\end{flushright}

\end{document}